\documentclass[twoside,reqno]{amsart}
\usepackage{amsmath}
\usepackage{amsfonts}
\usepackage{amsthm}
\usepackage{amssymb}
\usepackage{graphicx,psfrag,epsfig}
\usepackage{color}
\usepackage{mathrsfs}
\usepackage{pdfsync}
\usepackage{cite}
\usepackage{times}

\newcommand{\R}{\mathbb R}

\newcommand{\rd}{\mathrm{d}}
\newtheorem{theorem}{Theorem}[section]
\newtheorem{corollary}[theorem]{Corollary}
\newtheorem{lemma}[theorem]{Lemma}
\newtheorem{proposition}[theorem]{Proposition}

\numberwithin{equation}{section}
\newcommand{\e}{\varepsilon}
\begin{document}
\title{A variational approach  to a stationary free boundary problem modeling MEMS}
\author{Philippe Lauren\c cot}
\address{Institut de Math\'ematiques de Toulouse, UMR~5219, Universit\'e de Toulouse, CNRS, F--31062 Toulouse Cedex 9, France} 
\email{laurenco@math.univ-toulouse.fr}
\author{Christoph Walker}
\address{Leibniz Universit\"at Hannover, Institut f\" ur Angewandte Mathematik, Welfengarten 1, D--30167 Hannover, Germany} 
\email{walker@ifam.uni-hannover.de}
\date{\today}
%
\begin{abstract}
A variational approach is employed to find stationary solutions to a free boundary problem modeling an idealized electrostatically actuated MEMS device made of an elastic plate coated with a thin dielectric film and suspended above a rigid ground plate. The model couples a non-local fourth-order equation for the elastic plate deflection to the harmonic electrostatic potential in the free domain between the elastic and the ground plate. The corresponding energy is non-coercive reflecting an inherent singularity related to a possible touchdown of the elastic plate. Stationary solutions are constructed using a constrained minimization problem. A by-product is the existence of at least two stationary solutions for some values of the applied voltage.
\end{abstract}
%
\maketitle
%
%
%
\pagestyle{myheadings}
\markboth{\sc{Ph.  Lauren\c cot \& Ch. Walker}}{\sc{Variational approach  to a stationary MEMS model}}
%
%
\section{Introduction} \label{sec:int}

Microelectromechanical systems (MEMS) play a key r\^ole in many electronic devices nowadays and include micro-pumps, optical micro-switches, and sensors, to name but a few \cite{PB03}. Idealized electrostatically actuated MEMS consist of an elastic plate lying above a fixed ground plate and held clamped along its boundary. A Coulomb force induced by the application of a voltage difference across the device deflects the elastic plate. It is known from applications that a stable configuration is only obtained for voltage differences below a certain critical threshold as above this value the elastic plate may ``pull in'' on the ground plate. 

In a simplified and re-scaled geometry when presupposing zero variation in transversal direction (see Figure~\ref{MEMSfig}), the stationary problem can be described as finding the plate deflection $u=u(x)\in (-1,\infty)$ on the interval $I:= (-1,1)$ according to
\begin{align}
\beta \partial_x^4 u(x) - \left(\tau+a\|\partial_x u\|_{L_2(I)}^2\right) \partial_x^2 u(x) & = - \lambda \left( \varepsilon^2 |\partial_x \psi(x,u(x))|^2 + |\partial_z \psi(x,u(x))|^2 \right)\, , \quad x\in I\ , \label{ueq} \\
u(\pm 1) = \partial_x u(\pm 1) & = 0\, , \label{ubc}
\end{align}
along with the electrostatic potential $\psi=\psi(x,z)$ satisfying
\begin{eqnarray}
\varepsilon^2 \partial_x^2 \psi + \partial_z^2 \psi & = & 0\ , \quad (x,z)\in \Omega(u)\ , \label{psieq} \\
\psi(x,z) & = & \frac{1+z}{1+u(x)}\ ,\quad  (x,z)\in \partial\Omega(u)\ , \label{psibc}
\end{eqnarray}
in the region
\begin{equation*}
\Omega(u) := \left\{ (x,z)\in I\times\mathbb{R}\ :\ -1 < z < u(x) \right\} 
\end{equation*}
between the two plates. In equation \eqref{ueq}, the fourth-order term $\beta \partial_x^4 u$ with $\beta>0$ reflects plate bending while the linear second-order term $\tau \partial_x^2 u$ with $\tau\ge 0$ and the non-local second-order term $a\|\partial_x u\|_{L_2(I)}^2 \partial_x^2 u$ with $a\ge 0$ and
$$ \|\partial_x u\|_{L_2(I)}^2 :=\int_{-1}^1\vert \partial_x u\vert^2\,\rd x
$$
account for external stretching and self-stretching forces generated by large oscillations, respectively. The right-hand side of \eqref{ueq} is due to  the electrostatic forces exerted on the elastic plate with parameter $\lambda>0$ proportional to the square of the applied voltage difference and the device's aspect ratio $\varepsilon >0$. The boundary conditions \eqref{ubc} mean that the elastic plate is clamped. According to \eqref{psieq}-\eqref{psibc}, the electrostatic potential is harmonic in the region $\Omega(u)$ enclosed by the two plates with value~1 on the elastic plate and value~0 on the ground plate. We refer the reader e.g. to  \cite{EGG10,PB03,LWxxx} and the references therein for more details on the derivation of the model.

\begin{figure}
\centering\includegraphics[scale=.6]{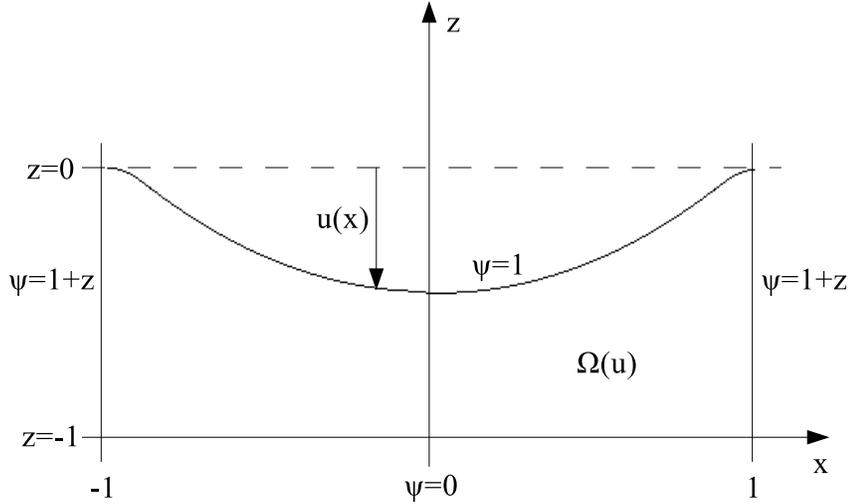}
\caption{\small Idealized electrostatic MEMS device.}\label{MEMSfig}
\end{figure}

A crucial feature of the model is the singularity arising in the term $\partial_z\psi(x,u(x))$ of \eqref{ueq} when $u(x)=-1$ (due to $\psi(x,-1)=0$ and $\psi(x,u(x))=1$), i.e. when the elastic plate touches down on the ground plate. The strength of this instability is in some sense tuned by the parameter $\lambda$ and it is thus expected that solutions to \eqref{ueq}-\eqref{psibc} only exist for small values of $\lambda$ below a certain threshold. Obviously, the stable operating conditions of MEMS devices and hence the existence of stationary solutions are of utmost importance in applications. Questions related to the pull-in threshold were the focus of a very active research in the recent past, however, almost exclusively dedicated to the simplified {\it small gap model} obtained by formally setting $\varepsilon =0$ in \eqref{ueq}-\eqref{psibc}. This reduces the problem to a singular nonlinear eigenvalue problem for $u$ of the form
\begin{equation}
\beta \partial_x^4 u(x) - \left(\tau+a\|\partial_x u\|_{L_2(I)}^2\right) \partial_x^2 u(x)  = - \lambda \frac{1}{(1+u(x))^2}\ ,\quad x\in I\ , \label{SG}
\end{equation}
subject to the boundary conditions \eqref{ubc} with explicitly given  electrostatic potential $$\psi(x,z)=\frac{1+z}{1+u(x)}\ .$$ For detailed results on the small gap model we refer the reader to \cite{EGG10,LWxxz} and the references therein in which also higher dimensional counterparts are investigated. Roughly speaking, in the one-dimensional (and two-dimensional radially symmetric) fourth-order small gap model with clamped boundary conditions and $a=0$ it is known \cite{LWxxz} that there is a threshold $\lambda_*>0$ such that there are (at least) two solutions to \eqref{SG} for $\lambda\in (0,\lambda_*)$, one solution for $\lambda=\lambda_*$, and no solution for $\lambda>\lambda_*$. 

A similar result one might expect also for the free boundary problem \eqref{ueq}-\eqref{psibc} with $\varepsilon >0$. A first step in this direction was made in \cite[Theorem 1.7]{LWxx}, where the following result was shown for $a=0$:

\begin{proposition}\label{00}
Let $a=0$.
\begin{itemize}
\item[(i)] There is $\lambda_s>0$ such that for each $\lambda\in (0,\lambda_s)$ there exists a solution $(U_\lambda,\Psi_\lambda)$ to \eqref{ueq}-\eqref{psibc} with $U_\lambda\in H^4(I)$ satisfying $-1<U_\lambda< 0$ in $I$ and  $\Psi_\lambda\in H^2(\Omega(U_\lambda))$. The mapping $\lambda\mapsto (\lambda,U_\lambda)$ defines a smooth curve in $\R\times H^4(I)$ with $U_\lambda\longrightarrow 0$ in $H^4(I)$ as $\lambda\to 0$.
\item[(ii)] There are $\varepsilon_*>0$ and ${\lambda_c}: (0,\varepsilon_*)\to (0,\infty)$ such that there is no solution $(u,\psi)$ to \eqref{ueq}-\eqref{psibc} for $\varepsilon\in (0,\varepsilon_*)$ and $\lambda>{\lambda_c}(\varepsilon)$.
\end{itemize}
\end{proposition}

Actually,  $(U_\lambda,\Psi_\lambda)$ for $\lambda\in (0,\lambda_s)$ is an asymptotically stable steady state for the corresponding dynamic problem. The proof of part~(i) of Theorem~\ref{0} is based on the Implicit Function Theorem and readily extends to the case $a>0$. For part~(ii) one may employ a nonlinear variant of the eigenfunction method involving a positive eigenfunction in $H^4(I)$ associated to the fourth-order operator $\beta\partial_x^4-\tau\partial_x^2$ subject to the clamped boundary condition \eqref{ubc}. For further use we now state the extension of Proposition~\ref{00} (i) to $a>0$.

\begin{theorem}\label{0}
Let $a\ge 0$.
There is $\lambda_s(a)>0$ such that for each $\lambda\in (0,\lambda_s(a))$ there exists a solution $(U_\lambda,\Psi_\lambda)$ to \eqref{ueq}-\eqref{psibc} with $U_\lambda\in H^4(I)$ satisfying $-1<U_\lambda< 0$ in $I$ and  $\Psi_\lambda\in H^2(\Omega(U_\lambda))$. The mapping $\lambda\mapsto (\lambda,U_\lambda)$ defines a smooth curve  in $\R\times H^4(I)$ with $U_\lambda\longrightarrow 0$ in $H^4(I)$ as $\lambda\to 0$.
\end{theorem}

Theorem~\ref{0} in particular ensures the existence of stationary solutions for small values of $\lambda$. However, it leaves open the question whether multiple solutions exist for such values of $\lambda$ which is a remarkable feature of the simplified small gap model as pointed out above. The purpose of the present paper is to give (partially) an affirmative answer. More precisely, we shall prove herein:

\begin{theorem}\label{T1}
For each $\rho>2$ there are $\lambda_\rho>0$, $u_\rho\in H^4(I)$, and $\psi_\rho\in H^2(\Omega(u_\rho))$ such that $(u_\rho,\psi_\rho)$ is a solution to \eqref{ueq}-\eqref{psibc} with $\lambda=\lambda_ \rho$. Both $u_\rho=u_\rho(x)$ and $\psi_\rho=\psi_\rho(x,z)$ are even with respect to $x\in I$ and $-1<u_\rho<0$ in $I$. Moreover, $\lambda_\rho\rightarrow 0$ as $\rho\rightarrow\infty$ and $u_\rho\not= U_{\lambda_\rho}$ for all $\rho>2$ sufficiently large.
\end{theorem}

Theorem~\ref{T1} provides multiple solutions to \eqref{ueq}-\eqref{psibc} for small values of $\lambda$ and is derived by a variational approach. It relies on the observation  that \eqref{ueq} is the Euler-Lagrange equation of the total energy $\mathcal{E}$ given by $\mathcal{E}(u):=\mathcal{E}_m(u)-\lambda \mathcal{E}_e(u)$ with mechanical energy
\begin{equation*}
\mathcal{E}_m(u) := \frac{\beta}{2} \|\partial_x^2 u \|_{L_2(I)}^2 + \frac{1}{2}\left(\tau + \frac{a}{2} \|\partial_x u\|_{L_2(I)}^2\right) \|\partial_x u \|_{L_2(I)}^2 
\end{equation*}
and electrostatic energy
\begin{equation*}
\mathcal{E}_e(u) := \int_{\Omega(u)} \left( \varepsilon^2 |\partial_x  \psi_u|^2 + |\partial_z  \psi_u|^2 \right)\ \mathrm{d}(x,z)\ ,
\end{equation*}
where the electrostatic potential $\psi_u$ is the solution to \eqref{psieq}-\eqref{psibc} associated to the given (sufficiently smooth) deflection $u$. Note that $\mathcal{E}$ is the sum of terms with different signs. The possible pull-in instability thus manifests in the non-coercivity of the energy $\mathcal{E}$, and due to this a plain minimization of the total energy is not appropriate. In fact, using Lemma~\ref{lea4}, it is not difficult to check that $\mathcal{E}$ is not bounded from below for $\lambda>0$ and we therefore take an alternative route and minimize the mechanical energy $\mathcal{E}_m$ constrained to (certain) deflections $u$ with fixed electrostatic energy $\mathcal{E}_e(u)=\rho$. Each minimizer $u_\rho$ of this constrained minimization problem together with the corresponding electrostatic potential $\psi_\rho:=\psi_{u_\rho}$ then yields a solution to \eqref{ueq}-\eqref{psibc} for the corresponding Lagrange multiplier $\lambda=\lambda_\rho$. Though lacking a continuity property with respect to $\rho>2$, the observation that $\mathcal{E}_e(U_\lambda)\rightarrow 2$ as $\lambda\rightarrow 0$ while $\lambda_\rho\rightarrow 0$ for $\mathcal{E}_e(u_\rho)=\rho\rightarrow \infty$ yields multiplicity of of solutions to \eqref{ueq}-\eqref{psibc} for small values of $\lambda$ in the sense that there is at least a sequence $\lambda_j\rightarrow 0$ of voltage values for which there are two different solutions $(u_j,\psi_j)$ (i.e. $\rho=j$ in Theorem~\ref{T1}) and $(U_{\lambda_j},\Psi_{\lambda_j})$ (i.e. $\lambda=\lambda_j$ in Theorem~\ref{0}). Note that, by taking a different sequence $\rho_j\rightarrow\infty$ with $\rho_j\not= j$, we obtain different solutions $(u_{\rho_j},\psi_{\rho_j})$ -- since the electrostatic energies differ -- but with possibly equal voltage values. We conjecture that, as in the simplified small gap model, the solutions constructed in Theorem~\ref{T1} actually lie on a smooth curve.

\medskip

To prove Theorem~\ref{T1} we first solve in Section~\ref{sec:elec} the elliptic problem \eqref{psieq}-\eqref{psibc} for the electrostatic potential $\psi=\psi_u$ for a given deflection $u$ and investigate then its dependence and that of the corresponding electrostatic energy $\mathcal{E}_e(u)$ with respect to $u$. Some technical details needed regarding continuity and differentiability properties of $\mathcal{E}_e$ and the right-hand side of \eqref{ueq} are postponed to Section~\ref{sec: 4}. The constrained minimization problem leading to Theorem~\ref{T1} is studied in Section~\ref{sec:mpc}.

\section{Some properties of the electrostatic energy and potential} \label{sec:elec}

We first focus on the elliptic problem \eqref{psieq}-\eqref{psibc} and investigate its solvability and properties of the corresponding electrostatic energy.

We shall use the following notation. To account for the clamped boundary conditions \eqref{ubc} we introduce, for $s\ge 0$ and $p\ge 2$,
$$
W_{p,D}^{s}(I):=\left\{\begin{array}{lll}
& \{v\in W_p^{s}(I)\,;\, v(\pm 1)=\partial_x v(\pm 1)=0\}\ , & s>\dfrac{3}{2}\ ,\\
& \{v\in W_p^{s}(I)\,;\, v(\pm 1)=0\}\ , & \dfrac{1}{2}<s<\dfrac{3}{2}\ ,\\
&   W_p^{s}(I)\ , & s<\dfrac{1}{2}\ ,
\end{array}
\right.
$$
and write $H_D^s(I):=W_{2,D}^s(I)$.
Similarly, $H^1_D(\Omega(u)):= \{v\in H^{1}(\Omega(u))\,;\, v=0 \ \text{ on }\ \partial\Omega\}$. For $s\ge 1$ we set 
\begin{equation*}
 S^{s} := \left\{ u \in H_{D}^{s}(I)\ : \ u > -1 \;\text{ on }\; I  \right\} \ , \qquad \mathcal{K}^s := \left\{ u \in H_D^s(I)\ :\ -1<u\le 0 \;\text{ on }\; I \right\}\ ,
\end{equation*}
and given $u\in S^1$ we define
\begin{equation}
b_u(x,z) := \left\{
\begin{array}{lcl}
\displaystyle{\frac{1+z}{1+u(x)}} & \text{ for } & (x,z)\in\overline{\Omega(u)}\ , \\
&  & \\
1 &  \text{ for } & (x,z)\in \overline{\Omega(0)}\setminus \overline{\Omega(u)}\ ,
\end{array}
\right. \label{a5}
\end{equation}
with $\Omega(0) = I \times (-1,0)$. Note that, if $u\in \mathcal{K}^1$, then the function $b_u$ belongs to $H^1(\Omega(0)) \cap C(\overline{\Omega(0)})$ which allows us to define $B_u\in H^{-1}(\Omega(0))$ (i.e. the dual space of $H_D^{1}(\Omega(0))$) by setting
\begin{equation}
\langle B_u , \vartheta \rangle := - \int_{\Omega(0)} \left[ \varepsilon^2 \partial_x b_u \partial_x \vartheta + \partial_z b_u \partial_z \vartheta \right]\ \mathrm{d}(x,z)\ , \quad \vartheta\in H^1_D(\Omega(0))\ . \label{a8}
\end{equation}

\subsection{Electrostatic potential} \label{sec:ep}

We first recall the existence and properties of weak solutions to \eqref{psieq}-\eqref{psibc} for $u\in \mathcal{K}^1$ which follow from \cite[Theorem~8.3]{GT01} and the Lax-Milgram Theorem.

\begin{lemma}\label{lea0}
Given $u\in S^1$, there is a unique weak solution $\psi_u\in H^1(\Omega(u))$ to \eqref{psieq}-\eqref{psibc} such that $\psi_u-b_u\in H_D^1(\Omega(u))$.  If, in addition, $u\in \mathcal{K}^1$, then $\psi_u-b_u$ satisfies the variational inequality
\begin{equation}
\begin{split}
\int_{\Omega(u)} \left( \varepsilon^2 |\partial_x(\psi_u-b_u)|^2 + |\partial_z(\psi_u-b_u)|^2 \right)\ \mathrm{d}(x,z) - 2 \langle B_u , \psi_u-b_u \rangle \\
\le \int_{\Omega(u)} \left( \varepsilon^2 |\partial_x\vartheta|^2 + |\partial_z\vartheta|^2 \right)\ \mathrm{d}(x,z) - 2 \langle B_u , \vartheta \rangle
\end{split} \label{a7}
\end{equation}
for all $\vartheta\in H^1_D(\Omega(u))$.
\end{lemma}

Replacing $\vartheta\in H_D^1(\Omega(u))$ in \eqref{a7} by $\xi-b_u$, where $\xi$ is an arbitrary function in $H^1(\Omega(u))$ satisfying $\xi-b_u\in H^1_D(\Omega(u))$, one easily obtains the following consequence:

\begin{lemma}\label{lea1}
Let $u\in \mathcal{K}^1$. For all $\xi\in H^1(\Omega(u))$ such that $\xi-b_u\in H^1_D(\Omega(u))$ there holds
\begin{equation}
\int_{\Omega(u)} \left( \varepsilon^2 |\partial_x \psi_u|^2 + |\partial_z \psi_u|^2 \right)\, \mathrm{d}(x,z) \le \int_{\Omega(u)} \left( \varepsilon^2 |\partial_x\xi|^2 + |\partial_z\xi|^2 \right)\, \mathrm{d}(x,z)\ . \label{a9}
\end{equation}
\end{lemma}

We collect additional properties of $\psi_u$ in the next result when $u$ is assumed to be more regular.

\begin{proposition}\label{lea5}
Let $\alpha\in [0,1/2)$. If $u\in S^{2-\alpha}$, then the weak solution $\psi_u$ to \eqref{psieq}-\eqref{psibc} belongs to $H^{2-\alpha}(\Omega(u))$. In addition, if $u\in\mathcal{K}^{2-\alpha}$, then
\begin{eqnarray}
1+z\ \le\ \psi_u(x,z) & \le & 1 \ , \quad (x,z)\in \Omega(u)\ , \label{a100} \\
\partial_x \psi_u(x,u(x)) & = & - \partial_z \psi_u(x,u(x))\ \partial_x u(x)\ , \quad x\in I\ , \label{a101} \\
\partial_z \psi_u(x,u(x)) & \ge & 0\ , \quad x\in I\ . \label{a102}
\end{eqnarray}
\end{proposition}

\begin{proof}
That $\psi_u\in H^{2-\alpha}(\Omega(u))$ for $u\in S^{2-\alpha}$ follows from Corollary~\ref{gth0} proved in Section~\ref{sec: 4}. Next, if $u\in\mathcal{K}^{2-\alpha}$, then owing to the non-positivity of $u$, the functions $(x,z)\mapsto 1+z$ and $(x,z)\mapsto 1$ are a subsolution and a supersolution to \eqref{psieq}-\eqref{psibc}, respectively, and \eqref{a100} follows from the comparison principle. To obtain \eqref{a101}, we simply differentiate the boundary condition $\psi_u(x,u(x))=1$, $x\in I$, with respect to $x$. Finally, \eqref{a102} is a straightforward consequence of the boundary condition $\psi_u(x,u(x))=1$, $x\in I$, and~\eqref{a100}. 
\end{proof}

Thanks to the continuity of the normal trace of the gradient from $H^{2-\alpha}(\Omega(u))$ to $H^{(1-2\alpha)/2}(I)$ for $\alpha\in [0,1/2)$ \cite[Theorem~1.5.2.1]{Gr85}, the regularity of the solution $\psi_u\in H^{2-\alpha}(\Omega(u))$ to \eqref{psieq}-\eqref{psibc} for $u\in S^{2-\alpha}$ provided by Proposition~\ref{lea5} gives a meaning to the right-hand side of \eqref{ueq}. We introduce the function $g$ by
\begin{equation}\label{gg}
g(u)(x) := \varepsilon^2 |\partial_x \psi_u(x,u(x))|^2 + |\partial_z \psi_u(x,u(x))|^2\ , \quad x\in I\ ,\quad u\in S^{2-\alpha}\ ,
\end{equation}
 and observe:

\begin{proposition}\label{lea5g}
If $\alpha\in [0,1/2)$, then $g\in C(S^{2-\alpha}, H^{\sigma}(I))$ for all $\sigma\in [0,1/2)$.
\end{proposition}

\begin{proof}
This is proved in Corollary~\ref{gth0}.
\end{proof}

\subsection{Electrostatic energy} \label{sec:ee}

We now study the properties of the electrostatic energy 
\begin{equation}
\mathcal{E}_e(u) = \int_{\Omega(u)} \left( \varepsilon^2 |\partial_x \psi_u|^2 + |\partial_z \psi_u|^2 \right)\, \mathrm{d}(x,z)\ ,\quad u\in S^1\ ,\label{a10}
\end{equation}
where $\psi_u\in H^1(\Omega(u))$ is provided by Lemma~\ref{lea0}. Alternatively, we may write for $u\in\mathcal{K}^1$
\begin{equation}
\begin{split}
\mathcal{E}_e(u) = & \int_{\Omega(u)} \left( \varepsilon^2 |\partial_x(\psi_u-b_u)|^2 + |\partial_z(\psi_u-b_u)|^2 \right)\, \mathrm{d}(x,z) \\ 
& - 2 \langle B_u , \psi_u - b_u \rangle + \int_{-1}^1 \left( 1 + \frac{\varepsilon^2}{3} |\partial_x u|^2 \right) \frac{\mathrm{d}x}{1+u}\ . 
\end{split} \label{a10a}
\end{equation}
We first establish a monotonicity property of $\mathcal{E}_e$ similar to \cite[Remarque~4.7.14]{HP05}.

\begin{proposition}\label{pra2}
Consider two functions $u_1$ and $u_2$ in $\mathcal{K}^1$ such that $u_1\le u_2$. Then $\mathcal{E}_e(u_2)\le \mathcal{E}_e(u_1)$. 
\end{proposition}

\begin{proof}
Consider $\xi\in H^1(\Omega(u_1))$ such that $\xi-b_{u_1}\in H_D^1(\Omega(u_1))$ and define
$$
\tilde{\xi}(x,z) := \left\{
\begin{array}{lcl}
\xi(x,z) & \text{for} & (x,z)\in\Omega(u_1)\ , \\
 & & \\
 1 & \text{for} & (x,z)\in\Omega(u_2)\setminus\overline{\Omega(u_1)}\ .
\end{array}
\right.
$$
Note that this definition is meaningful since $\Omega(u_1)\subset\Omega(u_2)$. Since $b_{u_1}(x,u_1(x)) = b_{u_2}(x,u_2(x)) = 1$ for $x\in I$, the previous construction guarantees that $\tilde{\xi}\in H^1(\Omega(u_2))$ with 
\begin{equation}
\tilde{\xi} - b_{u_2}\in H_D^1(\Omega(u_2)) \;\;\text{ and }\;\; \nabla\tilde{\xi} = \mathbf{1}_{\Omega(u_1)}\, \nabla\xi \ . \label{a11}
\end{equation}
We now infer from Lemma~\ref{lea1} and \eqref{a11} that
\begin{align*}
\mathcal{E}_e(u_2) \le & \int_{\Omega(u_2)} \left( \varepsilon^2 |\partial_x \tilde{\xi} |^2 + |\partial_z \tilde{\xi}|^2 \right)\ \mathrm{d}(x,z) \\ 
= & \int_{\Omega(u_1)} \left( \varepsilon^2 |\partial_x \xi |^2 + |\partial_z \xi|^2 \right)\ \mathrm{d}(x,z)\ .
\end{align*}
The above inequality being valid for all $\xi\in H^1(\Omega(u_1))$ satisfying $\xi-b_{u_1} \in H_D^1(\Omega(u_1))$, in particular for $\xi=\psi_{u_1}$, we conclude that $\mathcal{E}_e(u_2)\le \mathcal{E}_e(u_1)$. 
\end{proof}

We next turn to continuity and Fr\'echet differentiability of the functional $\mathcal{E}_e$.

\begin{proposition}\label{pra3}
If $\alpha\in [0,1/2)$, then $\mathcal{E}_e\in C(\mathcal{K}^1)\cap C^1(S^{2-\alpha})$ with $\partial_u\mathcal{E}_e(u) = -g(u)$ for $u\in S^{2-\alpha}$.
\end{proposition}

\begin{proof}
\textbf{Step~1: Continuity.} Let $(u_n)_{n\ge 1}$ be a sequence in $\mathcal{K}^1$ and $u\in\mathcal{K}^1$ such that $u_n\longrightarrow u$ in $H^1(I)$. We first observe that, for all $n\ge 1$, $\psi_{u_n} - b_{u_n}\in H_D^1(\Omega(u_n))$ is a weak solution to
\begin{equation}
\varepsilon^2 \partial_x^2 \left( \psi_{u_n} - b_{u_n} \right) + \partial_z^2 \left( \psi_{u_n} - b_{u_n} \right) = - B_{u_n} \ , \quad (x,z)\in \Omega(u_n)\ , \label{a12} 
\end{equation}
while the convergence of $(u_n)_{n\ge 1}$ toward $u$ in $H^1(I)$ entails that
\begin{equation}
\lim_{n\to\infty} \| B_{u_n} - B_u \|_{H^{-1}(\Omega(0))} = 0\ , \label{a13}
\end{equation}
where $\Omega(0) = I \times (-1,0)$.
Next, denoting the Hausdorff distance between open subsets of $\Omega(0)$ by $d_H$, see \cite[Section~2.2.3]{HP05} for instance, we realize that
$$
d_H(\Omega(u_n),\Omega(u)) \le \|u_n - u \|_{L_\infty(I)}\ ,
$$
and deduce from the continuous embedding of $H^1(I)$ in $L_\infty(I)$ that
\begin{equation}
\lim_{n\to\infty} d_H(\Omega(u_n),\Omega(u))  = 0\ . \label{a14}
\end{equation}
Since $\overline{\Omega(0)}\setminus\Omega(u_n)$ has a single connected component for all $n\ge 1$, it follows from \eqref{a12}, \eqref{a13}, \eqref{a14}, \cite[Theorem~4.1]{Sv93}, and \cite[Corollaire~3.2.6]{HP05} that 
\begin{equation}
\psi_{u_n} - b_{u_n} \longrightarrow \psi_u - b_u \;\;\text{ in }\;\; H_D^1(\Omega(0))\ .
\label{a140}
\end{equation}
Therefore, since 
$$
\lim_{n\to\infty} \int_{-1}^1 \left( 1 + \frac{\varepsilon^2}{3} |\partial_x u_n|^2 \right) \frac{\mathrm{d}x}{1+u_n} = \int_{-1}^1 \left( 1 + \frac{\varepsilon^2}{3} |\partial_x u|^2 \right) \frac{\mathrm{d}x}{1+u} 
$$
thanks to the continuous embedding of $H^1(I)$ in $L_\infty(I)$, we may pass to the limit as $n\to\infty$ in \eqref{a10a} for $u_n$ and use \eqref{a13} and \eqref{a140} to complete the proof. 

\smallskip

\textbf{Step~2: Differentiability.} Consider $u\in S^{2-\alpha}$ and $v\in H_D^{2-\alpha}(I)$. Owing to the continuous embedding of $H^{2-\alpha}(I)$ in $L_\infty(I)$, $u+s v$ still belongs to $S^{2-\alpha}$ for $s\in\mathbb{R}$ small enough and the map $s\mapsto \mathcal{E}_e(u+sv)$ is thus well-defined in a neighborhood of $s=0$. We then argue as in the proof of \cite[Proposition~2.2]{LWxx} with the help of a shape optimization approach (see \cite{HP05}, for instance) to show that this map is differentiable at $s=0$ with
\begin{equation*}
\frac{\rd}{\rd s}\mathcal{E}_e(u+sv)\Big|_{s=0} = - \int_{-1}^1 g(u) v\ \rd x\ . 
\end{equation*}
Consequently, $\mathcal{E}_e$ is G\^ateaux-differentiable with derivative $\partial_u\mathcal{E}_e(u)\in \mathcal{L}\left( H_D^{2-\alpha}(I) , \mathbb{R} \right)$. Moreover, since $g\in C(S^{2-\alpha}, L_2(I))$ by Proposition~\ref{lea5g}, the G\^ateaux-derivative $\partial_u\mathcal{E}_e$ is continuous as a mapping from $S^{2-\alpha}$ to $\mathcal{L}\left( H_D^{2-\alpha}(I) , \mathbb{R} \right)$. The claim follows from \cite[Proposition 4.8]{Ze86}.
\end{proof}

We next derive additional properties of $\mathcal{E}_e$ and, in particular, the following lower and upper bounds which have been established in \cite[Lemma~7]{ELW13} and \cite[Lemma~5.4]{LWxx}, respectively.

\begin{lemma}\label{lea4}
For $u\in\mathcal{K}^1$,
$$
2 \le \int_{-1}^1 \frac{\mathrm{d}x}{1+u(x)} \le \mathcal{E}_e(u) \le \int_{-1}^1 \left( 1 + \varepsilon^2 |\partial_x u(x)|^2 \right) \frac{\mathrm{d}x}{1+u(x)}\ .
$$
\end{lemma}

\begin{proof}
We recall the proof for the sake of completeness. We first deduce from \eqref{psibc} and the Cauchy-Schwarz inequality that, for $x\in I$,
\begin{align*}
\frac{1}{1+u(x)} & = \frac{\left( \psi_u(x,u(x)) - \psi_u(x,-1) \right)^2}{1+u(x)} = \frac{1}{1+u(x)} \left( \int_{-1}^{u(x)} \partial_z \psi_u(x,z)\ \mathrm{d}z \right)^2 \\
& \le \int_{-1}^{u(x)} \left( \partial_z \psi_u(x,z) \right)^2\ \mathrm{d}z \ . 
\end{align*}
Integrating the above inequality with respect to $x\in I$ readily gives the first inequality of Lemma~\ref{lea4}. We next infer from Lemma~\ref{lea1} with $\xi=b_u$, the latter being defined in \eqref{a5}, that
\begin{align*}
\mathcal{E}_e(u) & \le \int_{\Omega(u)} \left( \varepsilon^2 |\partial_x b_u|^2 + |\partial_z b_u|^2 \right)\ \mathrm{d}(x,z) \\ & \le \int_{\Omega(u)} \left[ \varepsilon^2 \frac{(1+z)^2}{(1+u(x))^4} |\partial_x u(x)|^2 + \frac{1}{(1+u(x))^2} \right]\ \mathrm{d} (x,z) \ ,
\end{align*}
from which the second inequality of Lemma~\ref{lea4} follows.
\end{proof}

Finally we recall the existence of a non-positive eigenfunction of the linear operator $\beta\partial_x^4 - \tau \partial_x^2\in \mathcal{L}(H_D^4(I),L_2(I))$ along with some of its properties.

\begin{lemma}\label{lea4.5}
\begin{itemize}
\item[(i)] The linear operator $\beta\partial_x^4 - \tau \partial_x^2\in \mathcal{L}(H_D^4(I),L_2(I))$ has a non-positive eigenfunction $\varphi_1\in H_D^4(I)\cap C^ \infty([-1,1])$ associated to a positive eigenvalue $\mu_1$. Moreover, $\varphi_1$ is even and it can be chosen such that $\varphi_1<0$ in $I$ with $\min_{[-1,1]} \varphi_1=-1$.
\item[(ii)] Given $\rho\in (2,\infty)$, there is $\eta_\rho\in (0,1)$ such that $\mathcal{E}_e(\eta_\rho\varphi_1) = \rho$ and $\eta_\rho\to 0$ as $\rho\to 2$.
\end{itemize}
\end{lemma}

\begin{proof}
Part~(i) follows from \cite[Theorem~4.7]{LWxy}, which is a consequence of the version of Boggio's principle \cite{Bo05} established in \cite{Gr02, LWxy, Ow97}. 
As for part~(ii), note that $\eta\varphi_1\in \mathcal{K}^1$ for $\eta\in [0,1)$ and
\begin{equation}
J(\eta) := \mathcal{E}_e(\eta\varphi_1) \ge \int_{-1}^1 \frac{\mathrm{d}x}{1+\eta\varphi_1(x)}\ , \quad \eta\in [0,1)\ , \label{a15}
\end{equation}
by Lemma~\ref{lea4}. We infer from Proposition~\ref{pra2} and Proposition~\ref{pra3} that $J$ is a non-decreasing and continuous function on $[0,1)$ with $J(0)=2$. In addition, $\varphi_1$ reaches necessarily its minimum $-1$ at some $x_0\in I$ and thus satisfies $\varphi_1(x_0)=-1$ and $\partial_x \varphi_1(x_0)=0$. Therefore, 
$$
0 \le 1+\varphi_1(x) \le \|\partial_x^2 \varphi_1\|_{L_\infty(I)}\ |x-x_0|^2 \;\;\text{ as }\;\; x\to x_0\ ,
$$ 
which implies that $(1+\varphi_1)^{-1}\not\in L_1(I)$. This property along with \eqref{a15} entails that $J(\eta)\to\infty$ as $\eta\to 1$. Recalling the continuity of $J$, we have thus shown that $[2,\infty)$ equals the range of $J$. The existence of $\eta_\rho$ for each $\rho\in (2,\infty)$ such that $\mathcal{E}_e(\eta_\rho\varphi_1) = \rho$ now follows. That $\eta_\rho\to 0$ as $\rho\to 2$ is a consequence of the fact that \eqref{a15} implies $J(\eta)=2$ if and only if $\eta=0$. 
\end{proof}

\section{A minimization problem with constraint}\label{sec:mpc}

Recall that, for $u\in H_D^2(I)$, the mechanical energy $\mathcal{E}_m$ is given by
\begin{equation*}
\mathcal{E}_m(u) = \frac{\beta}{2} \|\partial_x^2 u \|_{L_2(I)}^2 + \frac{1}{2}\left(\tau + \frac{a}{2}  \|\partial_x u\|_{L_2(I)}^2\right) \|\partial_x u \|_{L_2(I)}^2\ . 
\end{equation*}
Our goal is now to minimize $\mathcal{E}_m$ on the set
\begin{equation*}
\mathcal{A}_\rho := \left\{ u \in \mathcal{K}^2\ ;\ u \;\text{ is even and}\; \mathcal{E}_e(u)=\rho \right\} 
\end{equation*}
for a given $\rho\in (2,\infty)$. Note that $\mathcal{A}_\rho$ is non-empty as it contains $\eta_\rho \varphi_1$ according to Lemma~\ref{lea4.5}. We set
\begin{equation*}
\mu(\rho) := \inf_{u\in \mathcal{A}_\rho} \mathcal{E}_m(u) \ge 0
\end{equation*} 
 and first collect some properties of the function $\rho\mapsto \mu(\rho)$.

\begin{proposition}\label{leb3}
The function $\mu$ is non-decreasing on $(2,\infty)$ with
$$
\lim_{\rho\to 2} \mu(\rho)=0 \;\;\text{ and }\;\; \mu_\infty := \lim_{\rho\to \infty} \mu(\rho)  < \infty\ .
$$
\end{proposition}

\begin{proof}
Let $\rho\in (2,\infty)$. Since $\eta_\rho\varphi_1\in \mathcal{A}_\rho$ is an eigenfunction of the linear operator $\beta\partial_x^4 - \tau \partial_x^2$ associated to the eigenvalue $\mu_1$ and since $\eta_\rho^2<1$, a straightforward computation gives
$$
0 \le \mu(\rho) \le \mathcal{E}_m(\eta_\rho\varphi_1)  \le \eta_\rho^2 \mathcal{E}_m(\varphi_1)\ .
$$
Since $\mathcal{E}_m(\varphi_1)$ is finite, $\eta_\rho\in (0,1)$, and $\eta_\rho\to 0$ as $\rho\to 2$ by Lemma~\ref{lea4.5}, we readily obtain 
\begin{equation}
\lim_{\rho\to 2} \mu(\rho) = 0 \;\;\text{ and }\;\; 0 \le \mu(\rho) \le \mathcal{E}_m(\varphi_1) \ . \label{b12}
\end{equation}
Let us now check the monotonicity of $\mu$. To this end, fix $2<\rho_1<\rho_2$ and $v\in\mathcal{A}_{\rho_2}$. For all $t\in [0,1]$, the function $tv$ belongs to $\mathcal{K}^2$, and Proposition~\ref{pra2} and Proposition~\ref{pra3} imply that the function $h: [0,1]\to\mathbb{R}$, defined by $h(t) := \mathcal{E}_e(tv)$, is continuous and non-decreasing with $h(0)=2$ and $h(1)=\rho_2$. Since $\rho_1\in (2,\rho_2)$, there is $t_1\in (0,1)$ such that $h(t_1)=\rho_1$, that is, $t_1 v\in \mathcal{A}_{\rho_1}$. Consequently,
$$
\mu(\rho_1) \le \mathcal{E}_m(t_1 v) \le \mathcal{E}_m(v)\ .
$$
As $v$ was arbitrarily chosen in $\mathcal{A}_{\rho_2}$, the above inequality allows us to conclude that $\mu(\rho_1)\le \mu(\rho_2)$. Thus, $\mu$ is a non-decreasing function on $(2,\infty)$ which is bounded from above by $\mathcal{E}_m(\varphi_1) $ according to \eqref{b12}. It then has a finite limit $\mu_\infty\in [0,\mathcal{E}_m(\varphi_1) ]$ as $\rho\to\infty$. 
\end{proof}

We next show the existence of $u_\rho\in \mathcal{A}_\rho$ such that
\begin{equation}
\mathcal{E}_m(u_\rho) = \mu(\rho)\ , \label{b3}
\end{equation}
that is, $u_\rho$ is a minimizer of $\mathcal{E}_m$ in $\mathcal{A}_\rho$.

\begin{proposition}\label{thb1}
For each $\rho\in (2,\infty)$, there is at least one solution $u_\rho\in\mathcal{A}_\rho$ to the minimization problem~\eqref{b3}. 
\end{proposition}

The first step of the proof of Proposition~\ref{thb1} is a pointwise lower bound for functions in $\mathcal{A}_\rho$.

\begin{lemma}\label{leb2}
Given $\rho>2$ and $v\in\mathcal{A}_\rho$, assume that there is $K\ge 2/\rho$ such that $\|\partial_x^2 v\|_{L_2(I)} \le K$. Then
$$
\min_{[-1,1]} v \ge \frac{1}{\rho^3 K^2} - 1\ .
$$
\end{lemma}

\begin{proof}
Thanks to the continuous embedding of $H_D^2(I)$ in $C^1([-1,1])$, the function $v$ reaches its minimum $m$ at some point $x_m\in [-1,1]$. Since $\mathcal{E}_e(v)=\rho > 2$ and $v\in\mathcal{K}^2$, we realize that $v\not\equiv 0$ and $m\in (-1,0)$ so that $x_m\in I$. Therefore, $\partial_x v(x_m) =0$ and we may assume that $x_m\in [0,1)$ since $v$ is even. Using Taylor's expansion and H\"older's inequality, we find, for $x\in I$,
\begin{align}
v(x) & = m - \int_{x_m}^x (y-x) \partial_x^2 v(y)\ \mathrm{d}y \le m + \frac{|x-x_m|^{3/2}}{\sqrt{3}} \|\partial_x^2 v\|_{L_2(I)} \nonumber \\
& \le m + K |x-x_m|^{3/2}\ . \label{b10}
\end{align}
Next, since $v\in\mathcal{A}_\rho$, we infer from Lemma~\ref{lea4} and \eqref{b10} that
\begin{equation}
\rho = \mathcal{E}_e(v) \ge \int_{-1}^1 \frac{\mathrm{d}x}{1+v(x)} = 2 \int_0^1 \frac{\mathrm{d}x}{1+v(x)} \ge 2 \int_0^1 \frac{\mathrm{d}x}{1+m + K |x-x_m|^{3/2}}\ . \label{b11}
\end{equation}
If $x_m\in [1/2,1)$, then $x_m-(\rho K)^{-2} >0$, and it follows from \eqref{b11} that
$$
\rho \ge 2 \int_{x_m - (\rho K)^{-2}}^{x_m} \frac{\mathrm{d}x}{1+m + K |x-x_m|^{3/2}} \ge \frac{2 (\rho K)^{-2}}{1+m + K (\rho K)^{-3}}\ ,
$$
hence $m\ge \rho^{-3} K^{-2} -1$ as claimed. If $x_m\in [0,1/2)$, then $x_m+(\rho K)^{-2} <1$, and we deduce from \eqref{b11} that 
$$
\rho \ge 2 \int_{x_m}^{x_m+ (\rho K)^{-2}} \frac{\mathrm{d}x}{1+m + K |x-x_m|^{3/2}} \ge \frac{2 (\rho K)^{-2}}{1+m + K (\rho K)^{-3}}\ ,
$$
and the same computation as in the previous case completes the proof. 
\end{proof}

\medskip

\begin{proof}[Proof of Proposition~\ref{thb1}]
Let $(u_k)_{k\ge 1}$ be a minimizing sequence of $\mathcal{E}_m$ in $\mathcal{A}_\rho$ satisfying 
\begin{equation}
\mu(\rho) \le \mathcal{E}_m(u_k) \le \frac{k+1}{k} \mu(\rho)\ . \label{b13}
\end{equation}
A first consequence of Proposition~\ref{leb3} and \eqref{b13} is that $\|\partial_x^2 u_k\|_{L_2(I)}^2 \le 4\mu_\infty /\beta$ for all $k\ge 1$. Together with Lemma~\ref{leb2} (with $K=(2/\rho)+2\sqrt{\mu_\infty /\beta}$) this property ensures 
\begin{equation}
0 \ge u_k(x) \ge \frac{\beta}{8\rho (\beta + \mu_\infty  \rho^2)} - 1\ , \quad x\in [-1,1]\ , \quad k\ge 1\ . \label{b14}
\end{equation}
Also, owing to \eqref{b12}, \eqref{b13}, and Poincar\'e's inequality, the sequence $(u_k)_{k\ge 1}$ is bounded in $H^2_D(I)$ and thus relatively compact in $C^1([-1,1])$. Consequently, there are $u\in H_D^2(I)$ and a subsequence of $(u_k)_{k\ge 1}$ (not relabeled) such that 
\begin{equation}
\begin{split}
u_k \longrightarrow u & \;\;\text{ in }\;\; C^1([-1,1])\ , \\
u_k \rightharpoonup u & \;\;\text{ in }\;\; H_D^2(I)\ .
\end{split} \label{b15}
\end{equation}
Combining \eqref{b14} and \eqref{b15} we conclude that 
\begin{equation*}
0 \ge u(x) \ge \frac{\beta}{8\rho (\beta +  \mu_\infty \rho^2)} - 1\ , \quad x\in [-1,1]\ ,
\end{equation*}
hence $u\in\mathcal{K}^2$. We then infer from Proposition~\ref{pra3} that
$$
\mathcal{E}_e(u) = \lim_{k\to\infty} \mathcal{E}_e(u_k) = \rho\ ,
$$
and so $u\in\mathcal{A}_\rho$. Since
$$
\mathcal{E}_m(u) \le \liminf_{k\to\infty} \mathcal{E}_m(u_k) \le \mu(\rho)
$$
by \eqref{b13} and \eqref{b15}, we deduce that $\mathcal{E}_m(u)=\mu(\rho)$ so that $u$ is a minimizer of $\mathcal{E}_m$ in $\mathcal{A}_\rho$. 
\end{proof}

\begin{theorem}\label{thb2}
Consider $\rho\in (2,\infty)$ and let $u\in \mathcal{A}_\rho$ be an arbitrary minimizer of $\mathcal{E}_m$ in $\mathcal{A}_\rho$. Then $u\in H_D^4(I)$ and there is $\lambda_u>0$ such that
\begin{equation}
\beta \partial_x^4 u(x) - \left(\tau + a \|\partial_x u\|_{L_2(I)}^2\right) \partial_x^2 u(x)  =  - \lambda_u \left( \varepsilon^2 |\partial_x \psi_u(x,u(x))|^2 + |\partial_z \psi_u(x,u(x))|^2 \right) \label{b4a} 
\end{equation}
for $x\in I$, where $\psi_u\in H^2(\Omega(u))$ denotes the associated solution to \eqref{psieq}-\eqref{psibc} given by Lemma~\ref{lea0} and Proposition~\ref{lea5}. Furthermore, 
\begin{equation}
0 < \lambda_u \le \frac{  8 \mu_\infty \left( \sqrt{\beta} + \varepsilon^2 \sqrt{\mu_\infty} \right)}{\sqrt{\beta} (\rho-2)^2}\ . \label{b20}
\end{equation}
\end{theorem} 

\begin{proof}
Let $u\in \mathcal{A}_\rho\subset\mathcal{K}^2$ be a minimizer of $\mathcal{E}_m$. Recall from Proposition~\ref{pra3}
that the derivative of $\mathcal{E}_e$
is given by
$$
\langle \partial_u\mathcal{E}_e(u) , \vartheta \rangle = - \int_{-1}^1 g(u) \vartheta\ \mathrm{d}x\ ,\quad \vartheta\in H^2_D(I)\ ,
$$
with $g(u)\in L_2(I)$ while clearly
$$
\langle \partial_u\mathcal{E}_m(u) , \vartheta \rangle = \int_{-1}^1 \left(\beta\partial_x^2 u\ \partial_x^2\vartheta + \big(\tau+a \|\partial_x u\|_{L_2(I)}^2 \big) \partial_x u\ \partial_x\vartheta \right)\ \mathrm{d}x\ ,\quad \vartheta\in H^2_D(I)\ .
$$
Since $u$ solves \eqref{b3} and $g(u)$ is non-negative, \cite[4.14.Proposition  1] {Ze95} implies that there is a Lagrange multiplier $\lambda_u\in\mathbb{R}$ such that 
\begin{equation}
\langle \partial_u\mathcal{E}_m(u) , \vartheta \rangle = \lambda_u \langle \partial_u\mathcal{E}_e(u) , \vartheta \rangle\ , \quad \vartheta\in H^2_D(I)\ .\label{b16}
\end{equation}
We may then combine \eqref{b16} and classical elliptic regularity to conclude that $u\in H_D^4(I)$ solves \eqref{b4a} in a strong sense. In addition, taking $\vartheta=u$ in \eqref{b16} gives
\begin{equation}
\beta \|\partial_x^2 u \|_{L_2(I)}^2 + \tau \|\partial_x u \|_{L_2(I)}^2 + a \|\partial_x u \|_{L_2(I)}^4  = - \lambda_u \int_{-1}^1 u g(u)\ \mathrm{d}x\ , \label{b200}
\end{equation}
hence $\lambda_u>0$ since $g(u)$ is non-negative and $u$ is non-positive and different from zero.

We are left with the upper bound \eqref{b20} on $\lambda_u$. On the one hand, multiplying \eqref{psieq} by $(1+u) \psi_u - (1+z)$, integrating over $\Omega(u)$, and using
$$
(1+u(x)) \psi_u(x,z) - (1+z) = 0\ , \quad (x,z)\in \partial\Omega(u)\ , 
$$
we obtain from Green's formula that
\begin{align*}
0 & = \int_{\Omega(u)} \left[ \varepsilon^2 \partial_x\psi_u \partial_x \left( (1+u) \psi_u \right) + \partial_z \psi_u \left( (1+u) \partial_z \psi_u - 1 \right) \right]\, \mathrm{d}(x,z) \\
& = \int_{\Omega(u)} \left[ (1+u) \left( \varepsilon^2 |\partial_x \psi_u|^2 + |\partial_z \psi_u|^2 \right) + \varepsilon^2 \psi_u \partial_x \psi_u\ \partial_x u \right]\, \mathrm{d}(x,z) - 2\ ,
\end{align*}
whence
\begin{equation}
\int_{\Omega(u)} \left[ u \left( \varepsilon^2 |\partial_x \psi_u|^2 + |\partial_z \psi_u|^2 \right) + \varepsilon^2 \psi_u \partial_x \psi_u\ \partial_x u \right]\, \mathrm{d}(x,z) = 2 - \mathcal{E}_e(u)\ . \label{b17}
\end{equation}
On the other hand, we multiply \eqref{psieq} by $u \psi_u$ and integrate over $\Omega(u)$. Using again Green's formula along with the values of $u$ and $\psi_u$ on the boundary of $\Omega(u)$, we find
\begin{align*}
0 = & - \int_{\Omega(u)} \left[ \varepsilon^2 \partial_x\psi_u \partial_x \left( u \psi_u \right) + \partial_z \psi_u \left( u \partial_z \psi_u \right) \right]\, \mathrm{d}(x,z) \\
& - \varepsilon^2 \int_{-1}^1 u(x) \partial_x u(x)\ \partial_x \psi_u(x,u(x))\, \mathrm{d}x + \int_{-1}^1 u(x) \partial_z \psi_u(x,u(x))\, \mathrm{d}x \\
= & - \int_{\Omega(u)} \left[ u \left( \varepsilon^2 |\partial_x\psi_u|^2 + |\partial_z \psi_u|^2 \right) + \varepsilon^2 \psi_u \partial_x \psi_u \, \partial_x u) \right]\, \mathrm{d}(x,z) \\
& + \int_{-1}^1 u(x) \left[ \partial_z \psi_u(x,u(x)) -  \varepsilon^2 \partial_x u(x)\, \partial_x \psi_u(x,u(x)) \right]\, \mathrm{d}x\ .
\end{align*}
Combining \eqref{b17} with the above identity and \eqref{a101} we end up with
\begin{equation}
- \int_{-1}^1 u(x) \left( 1 + \varepsilon^2 |\partial_x u(x)|^2 \right)\ \partial_z \psi_u(x,u(x))\ \mathrm{d}x = \mathcal{E}_e(u) - 2\ . \label{b19}
\end{equation} 
Now it follows from \eqref{b3}, \eqref{b200}, \eqref{b19}, Jensen's inequality, the bounds $-1<u\le 0$, and the non-negativity \eqref{a102} of $x\mapsto \partial_z \psi_u(x,u(x))$ that
\begin{align*}
4 \mu(\rho) &  \ge \beta \|\partial_x^2 u \|_{L_2(I)}^2 + \tau \|\partial_x u \|_{L_2(I)}^2 + a \|\partial_x u \|_{L_2(I)}^4  \\
& = -\lambda_u \int_{-1}^1 u(x) \left( 1 + \varepsilon^2 |\partial_x u(x)|^2 \right)\ |\partial_z \psi_u(x,u(x))|^2\ \mathrm{d}x \\
& \ge \lambda_u\ \frac{\displaystyle{\left( \int_{-1}^1 |u(x)| \left( 1 + \varepsilon^2 |\partial_x u(x)|^2 \right)\ \partial_z \psi_u(x,u(x))\ \mathrm{d}x \right)^2}}{\displaystyle{\int_{-1}^1 |u(x)| \left( 1 + \varepsilon^2 |\partial_x u(x)|^2 \right)\ \mathrm{d}x}} \\
& \ge \lambda_u\ \frac{(\mathcal{E}_e(u)-2)^2}{2+\varepsilon^2 \|\partial_x u\|_{L_2(I)}^2}\ .
\end{align*}
We finally observe that $\mathcal{E}_e(u)=\rho$ as $u\in\mathcal{A}_\rho$ while
$$
\|\partial_x u\|_{L_2(I)}^2 = - \int_{-1}^1 u\ \partial_x^2 u\ \mathrm{d}x \le \int_{-1}^1 |\partial_x^2 u|\ \mathrm{d}x \le \sqrt{2} \|\partial_x^2 u\|_{L_2(I)} \le 2\sqrt{\frac{\mu(\rho)}{\beta}}\ ,
$$
since $u\in \mathcal{K}^2$ solves \eqref{b3}. Therefore,
$$
4 \mu(\rho) \ge \lambda_u\ \frac{\sqrt{\beta} (\rho-2)^2}{2\left( \sqrt{\beta} +\varepsilon^2 \sqrt{\mu(\rho)} \right)}\ ,
$$
which gives \eqref{b20} after using Proposition~\ref{leb3}.
\end{proof}

\begin{proof}[Proof of Theorem~\ref{T1}] Clearly, Proposition~\ref{thb1} and Theorem~\ref{thb2} imply that for each $\rho>2$ there are $\lambda_\rho>0$, $u_\rho\in H_D^4(I)$, and $\psi_\rho\in H^2(\Omega(u_\rho))$ such that $(u_\rho,\psi_\rho)$ is a solution to \eqref{ueq}-\eqref{psibc} with $\lambda=\lambda_ \rho$.  We recall that  $\lambda\mapsto (\lambda,U_\lambda)$ defines a smooth curve  in $\R\times H^4(I)$ starting at $(0,0)$ according to Theorem~\ref{0} so that $\mathcal{E}_e(U_\lambda)\rightarrow 2$ as $\lambda\rightarrow 0$ due to Proposition~\ref{pra3}. Consequently, since $\mathcal{E}_e(u_\rho)=\rho$ and $\lambda_\rho\to 0$ as $\rho\to\infty$, we realize that $u_\rho\not= U_{\lambda_\rho}$ for large~$\rho$. Finally, since $u_\rho$ is even and uniquely determines $\psi_\rho$, it readily follows that $\psi_\rho=\psi_\rho(x,z)$ is even with respect to $x\in I$.
\end{proof}

\section{Regularity of solutions to \eqref{psieq}-\eqref{psibc}}\label{sec: 4}

In this section we provide the technical proofs of Proposition~\ref{lea5} and Proposition~\ref{lea5g} that were postponed. That is, we shall improve the regularity of the weak solution $\psi_u$ to \eqref{psieq}-\eqref{psibc} given in Lemma~\ref{lea0} for smoother deflection $u$ and prove continuity properties of the function $g$ defined in \eqref{gg}. In order to do so we introduce the transformation
$$
T_u(x,z):=\left(x,\frac{1+z}{1+u(x)}\right)\ ,\quad (x,z)\in {\Omega(u)}\ ,
$$
mapping $\Omega(u)$ onto the fixed rectangle \mbox{$\Omega:=I\times (0,1)$}. We then transform the elliptic problem \eqref{psieq}-\eqref{psibc} for $\psi_u$ in the variables $(x,z)\in\Omega(u)$ to the elliptic problem
\begin{equation}
\mathcal{L}_u \Phi_u = f_u \;\;\text{ in }\;\; \Omega\ , \qquad \Phi_u=0 \;\;\text{ on }\;\; \partial\Omega\ , \label{g2}
\end{equation}
for  $\Phi_u(x,\eta)=\psi_u\circ T_u^{-1}(x,\eta)-\eta$ in the variables $(x,\eta)=T_u(x,z)\in \Omega$, where the operator $\mathcal{L}_u$ is given by
\begin{equation}
\begin{split}
\mathcal{L}_u w\, :=\, & \e^2\ \partial_x^2 w - 2\e^2\ \eta\ \frac{\partial_x u(x)}{1+u(x)}\ \partial_x\partial_\eta w
+ \frac{1+\e^2\eta^2(\partial_x u(x))^2}{(1+u(x))^2}\ \partial_\eta^2 w\\
& + \e^2\ \eta\ \left[ 2\ \left(\frac{\partial_x u(x)}{1+u(x)} \right)^2 - \frac{\partial_x^2 u(x)}{1+u(x)} \right]\ \partial_\eta w
\end{split} \label{acdc}
\end{equation}
and the right-hand side $f_u$ is given by
\begin{equation}\label{ff}
f_u(x,\eta) := \varepsilon^2 \eta \left[ \partial_x \left(\frac{\partial_x u(x)}{1+u(x)}\right) - \left(\frac{\partial_x u(x)}{1+u(x)}\right)^2 \right]\ , \quad (x,\eta)\in \Omega\ .
\end{equation}
The goal is then to obtain uniform estimates for $\Phi_u$ in the anisotropic space
\begin{equation*}
X(\Omega) := \{ w \in H^1(\Omega)\ ;\ \partial_\eta w \in H^1(\Omega)\} 
\end{equation*}
in dependence of deflections $u$ belonging to certain open subsets
$$
S_p^{s}(\kappa) := \left\{ u \in W_{p,D}^{s}(I)\ ; \ u > -1+ \kappa \;\text{ in }\; I \;\text{ and }\; \|u\|_{W_{p,D}^s(I)} < \frac{1}{\kappa} \right\}
$$
of $W_{p,D}^{s}(I)$, where $p\ge 2$, $s>1/p$, and $\kappa\in (0,1)$. Note that the closure of $S_p^s(\kappa)$ in $W_{p,D}^s(I)$ is 
$$
\overline{S}_p^{s}(\kappa) = \left\{ u \in W_{p,D}^{s}(I)\ ; \ u \ge -1+ \kappa \;\text{ in }\; I \;\text{ and }\; \|u\|_{W_{p,D}^s(I)} \le \frac{1}{\kappa} \right\}
$$
and  $S^s= \cup_{\kappa\in (0,1)} S_2^s(\kappa)$. More precisely, we shall prove the following result regarding the problem \eqref{g2}:

\begin{proposition}\label{gth1}
Let $\alpha\in [0,1/2)$, $\nu\in (\alpha,1/2)$, $\kappa\in (0,1)$, and $u\in \overline{S}_2^{2-\alpha}(\kappa)$. There is a unique solution $\Phi_u\in X(\Omega)\cap H^{2-\nu}(\Omega)$ to \eqref{g2}
which satisfies
\begin{equation}
\|\Phi_u\|_{X(\Omega)} + \|\Phi_u\|_{H^{2-\nu}(\Omega)} \le c_1(\kappa) \label{g3} 
\end{equation}
for some positive constant $c_1(\kappa)$ depending only on $\varepsilon$, $\alpha$, $\nu$, and $\kappa$. In addition, the distribution $q_u$, defined for $\vartheta\in C_0^\infty(\Omega)$ by
\begin{eqnarray}
\langle q_u , \vartheta \rangle & := & - \int_\Omega \left[ \partial_x  \Phi_u(x,\eta) - \eta U(x) \partial_\eta \Phi_u(x,\eta) \right] \partial_x\vartheta(x,\eta)\ \rd(x,\eta) \nonumber \\
&&  \ + \int_\Omega \eta U(x) \partial_x\partial_\eta \Phi_u(x,\eta) \vartheta(x,\eta)\ \rd(x,\eta)\label{g1a}
\end{eqnarray}
with $U := \partial_x\ln{(1+u)}$, belongs to the dual space $H^{-\alpha}(\Omega)$ of $H^{\alpha}(\Omega)$, and there is $c_2(\kappa)$ depending only on $\varepsilon$, $\alpha$, and $\kappa$ such that 
\begin{equation}
\|q_u\|_{H^{-\alpha}(\Omega)} \le c_2(\kappa)\ . \label{g3a}
\end{equation}
Furthermore, if $(u_n)_{n\ge 1}$ is a sequence in $\overline{S}_2^{2-\alpha}(\kappa)$ converging weakly in $H^{2-\alpha}(I)$ toward $u\in \overline{S}_2^{2-\alpha}(\kappa)$, then
\begin{equation}
\Phi_{u_n} \rightharpoonup \Phi_u \;\;\text{ in }\;\; X(\Omega)\cap H^{2-\nu}(\Omega) \label{g4}
\end{equation}
and $(\Phi_{u_n})_{n\ge 1}$ converges strongly to $\Phi_u$ in $H^1(\Omega)$.
\end{proposition}

The proof of Proposition~\ref{gth1} requires several steps which will be given in the next subsection, the actual proof  of Proposition~\ref{gth1} being contained in Subsection~\ref{subsec:42}. From Proposition~\ref{gth1} we may in particular derive more regularity for the solution $\psi_u$ to \eqref{psieq}-\eqref{psibc} and the continuity of the function $g$ defined in \eqref{gg} as stated in the next corollary.

\begin{corollary}\label{gth0}
Given $\alpha\in [0,1/2)$ and $u\in S^{2-\alpha}$, the corresponding solution $\psi_u$ to \eqref{psieq}-\eqref{psibc} belongs to $H^{2-\alpha}(\Omega(u))$. In addition, $g\in C(S^{2-\alpha}, H^\sigma(I))$ for all $\sigma\in [0,1/2)$.
\end{corollary}

 As already indicated, Proposition~\ref{lea5} and Proposition~\ref{lea5g} are now consequences of Corollary~\ref{gth0} which is proved in Subsection~\ref{subsec:42}.

\subsection{Auxiliary Results}

The starting point for the proof of Proposition~\ref{gth1} is the solvability of the Dirichlet problem for $\mathcal{L}_u$ in $H^{-1}(\Omega)$ for $u\in \overline{S}_2^{2-\alpha}(\kappa)$ and in $L_2(\Omega)$ for $u\in \overline{S}_p^2(\kappa)$ with $p>2$. 

\begin{lemma}\label{gle2}
Let $\alpha\in [0,1/2)$, $p>2$, and $\kappa\in (0,1)$.
\begin{itemize}
\item[(i)] Given $u\in \overline{S}_2^{2-\alpha}(\kappa)$ and $h\in  H^{-1}(\Omega)$ there is a unique weak solution $\Phi\in H_D^1(\Omega)$ to
\begin{equation}
\mathcal{L}_u \Phi = h \;\;\text{ in }\;\; \Omega\ , \quad \Phi=0 \;\;\text{ on }\;\; \partial\Omega\ . \label{g5}
\end{equation}
Moreover, there is $c_3(\kappa)$ depending only on $\varepsilon$, $\alpha$, and $\kappa$ such that 
\begin{equation}
\|\Phi\|_{H^1(\Omega)} \le c_3(\kappa) \|h\|_{H_{D}^{-1}(\Omega)}\ . \label{g6}
\end{equation}

\item[(ii)] Given $u\in \overline{S}_p^{2}(\kappa)$ and $h\in L_2(\Omega)$ there is a unique solution $\Phi\in H_D^1(\Omega)\cap H^2(\Omega)$ to \eqref{g5}.
\end{itemize}
\end{lemma}

\begin{proof}
The proof of Lemma~\ref{gle2}~(i) is similar to that of the first statement of \cite[Lemma~2.2]{ELWxx} thanks to the continuous embedding of $H^{2-\alpha}(I)$ in $W_\infty^1(I)$. Next, Lemma~\ref{gle2}~(ii) follows from the second statement of \cite[Lemma~6]{ELW14}.
\end{proof}

We next provide continuity properties with respect to $u$ and $h$ of the solution $\Phi$ to \eqref{g5}.

\begin{lemma}\label{gle2c}
Let $\alpha\in [0,1/2)$ and $\kappa\in (0,1)$. Consider sequences $(u_n)_{n\ge 1}$ in $ \overline{S}_2^{2-\alpha}(\kappa)$ and $(h_n)_{n\ge 1}$ in $H^{-1}(\Omega)$ such that
$$
u_n \rightharpoonup u \;\text{ in }\; H^{2-\alpha}(I) \;\text{ and }\; h_n \rightharpoonup h \;\text{ in }\; H^{-1}(\Omega)\ . 
$$
Denoting the solution to \eqref{g5} with $(u_n,h_n)$ by $\Phi_n$ and that of \eqref{g5} by $\Phi$ there holds
$$
\Phi_n \rightharpoonup \Phi \;\text{ in }\; H^1(\Omega)\ .
$$
\end{lemma}

\begin{proof}
Let $n\ge 1$ and $\vartheta\in H_D^1(\Omega)$. Setting $U_n := \partial_x\ln{(1+u_n)}$, the weak formulation of \eqref{g5} for $\Phi_n$ reads
\begin{align}
 \varepsilon^2& \int_\Omega \left[ \partial_x \Phi_n - \eta U_n \partial_\eta \Phi_n \right] \partial_x \vartheta\ \rd(x,\eta) \nonumber \\
+ & \int_\Omega \left[ -\varepsilon^2 \eta U_n \partial_x\Phi_n + \left( \frac{1}{(1+u_n)^2} + \varepsilon^2 \eta^2 U_n^2 \right) \partial_\eta \Phi_n \right] \partial_\eta \vartheta\ \rd(x,\eta) \nonumber \\
- & \varepsilon^2 \int_\Omega \left[ U_n \partial_x\Phi_n - \eta U_n^2 \partial_\eta \Phi_n \right] \vartheta\ \rd(x,\eta) = - \int_\Omega h_n \vartheta\ \rd(x,\eta)\ . \label{g7b}
\end{align}
Owing to the compactness of the embedding of $H^{2-\alpha}(I)$ in $W_\infty^1(I)$, there is a subsequence $(u_{n_k})_{k\ge 1}$ of $(u_n)_{n\ge 1}$ such that $(u_{n_k})_{k\ge 1}$ converges toward $u$ in $W_\infty^1(I)$ as $k\to\infty$. This implies in particular that $(U_{n_k})_{k\ge 1}$ and $(U_{n_k}^2)_{k\ge 1}$ converge, respectively, toward $U:=\partial_x\ln{(1+u)}$ and $U^2$ in $L_\infty(I)$. Furthermore, it follows from \eqref{g6} and the boundedness of $(h_n)_{n\ge 1}$ in $H^{-1}(\Omega)$ and that of $(u_n)_{n\ge 1}$ in $\overline{S}_2^{2-\alpha}(\kappa)$ that $(\Phi_n)_{n\ge 1}$ is bounded in $H_D^1(\Omega)$. We may therefore assume that $(\Phi_{n_k})_{k\ge 1}$ converges weakly toward some $\Psi$ in $H_D^1(\Omega)$. Combining the previous weak convergences we realize that all terms in \eqref{g7b} converge and letting $n_k\to\infty$ in \eqref{g7b} shows that $\Psi$ is a weak solution to \eqref{g5}. According to Lemma~\ref{gle2}~(i), $\Psi$ coincides with the unique solution $\Phi$ to \eqref{g5}. This, in turn, implies the convergence of the whole sequence $(\Phi_n)_{n\ge 1}$ and completes the proof.
\end{proof}

We next derive additional estimates on the solution to \eqref{g5} for some specific choices of the right-hand side $h$ and begin with the case $h\in L_2(\Omega)$.

\begin{lemma}\label{gle3a}
Let $\alpha\in [0,1/2)$, $\nu\in (\alpha,1/2)$, $\kappa\in (0,1)$, $u\in \overline{S}_2^{2-\alpha}(\kappa)$, and $h\in L_2(\Omega)$. The unique solution $\Phi$ to \eqref{g5}, given by Lemma~\ref{gle2}~(i), belongs to $X(\Omega)\cap H^{2-\nu}(\Omega)$, and there is $c_4(\kappa)$ depending only on $\varepsilon$, $\alpha$, $\nu$, and $\kappa$ such that 
\begin{equation}
\|\Phi\|_{X(\Omega)} + \|\Phi\|_{H^{2-\nu}(\Omega)} \le c_4(\kappa) \|h\|_{L_2(\Omega)}\ . \label{g8}
\end{equation}
Furthermore, the distribution $q$, defined for $\vartheta\in C_0^\infty(\Omega)$ by
\begin{eqnarray}
\langle q , \vartheta \rangle & := & - \int_\Omega \left[ \partial_x  \Phi(x,\eta) - \eta U(x) \partial_\eta \Phi(x,\eta) \right] \partial_x\vartheta(x,\eta)\ \rd(x,\eta) \nonumber \\
& & \ + \int_\Omega \eta U(x) \partial_x\partial_\eta \Phi(x,\eta) \vartheta(x,\eta)\ \rd(x,\eta) \label{g8a}
\end{eqnarray}
with $U:= \partial_x \ln{(1+u)}$, belongs to $L_2(\Omega)$, and there is $c_5(\kappa)$ depending only on $\varepsilon$, $\alpha$, and $\kappa$ such that 
\begin{equation}
\| q\|_{L_2(\Omega)} \le c_5(\kappa) \|h\|_{L_2(\Omega)}\ . \label{g8b}
\end{equation}
\end{lemma}

\begin{proof}
\noindent\textbf{Step~1:} We first assume that $u\in \overline{S}_2^{2-\alpha}(\kappa)\cap W_p^2(I)$ for some $p>2$. Clearly, there is $\kappa'\in (\kappa,1)$ such that $u\in\overline{S}_p^2(\kappa')$. Thus, by Lemma~\ref{gle2}~(ii), the solution $\Phi$ to \eqref{g5} belongs to $H^2(\Omega)$. Set $\zeta := \partial_\eta^2 \Phi$ and $\omega := \partial_x \partial_\eta \Phi$. We multiply \eqref{g5} by $\zeta$ and integrate over $\Omega$ to find
\begin{align*}
\int_\Omega h \zeta\ \mathrm{d}(x,\eta)  =\ & \varepsilon^2 \int_\Omega \partial_x^2 \Phi \partial_\eta^2 \Phi\ \mathrm{d}(x,\eta) - 2 \varepsilon^2 \int_\Omega \eta U \omega \zeta\ \rd(x,\eta) \\
& \ + \int_\Omega \left[ \frac{1}{(1+u)^2} + \varepsilon^2 \eta^2 U^2 \right] \zeta^2\ \rd(x,\eta) + \varepsilon^2 \int_\Omega \eta \left[ U^2 - \partial_x U \right] \zeta \partial_\eta \Phi\ \rd(x,\eta)\ .
\end{align*}
Using the identity 
$$
\int_\Omega \partial_x^2 \Phi \partial_\eta^2 \Phi\ \mathrm{d}(x,\eta)  = \int_\Omega \omega^2\ \mathrm{d}(x,\eta)
$$
from \cite[Lemma~4.3.1.2 \&~4.3.1.3]{Gr85} we deduce
\begin{equation}
\varepsilon^2 \|\omega - \eta U \zeta \|_{L_2(\Omega)}^2 + \left\| \frac{\zeta}{1+u} \right\|_{L_2(\Omega)}^2 = \frac{\varepsilon^2}{2} R_1 + R_2 \label{g10}
\end{equation}
with
\begin{eqnarray}
R_1 & := & 2 \int_\Omega \eta (\partial_x U - U^2) \partial_\eta \Phi \partial_\eta^2 \Phi\ \rd(x,\eta)\ , \label{g11} \\
R_2 & := & \int_\Omega h \zeta \rd(x,\eta)\ . \label{g12}
\end{eqnarray}
Introducing the trace $\gamma(x) := \partial_\eta \Phi(x,1)$ for $x\in I$, we infer from Green's formula and $U(\pm 1)=0$ that 
\begin{align*}
R_1 & =\int_{-1}^1  (\partial_x U - U^2) \gamma^2\ \rd x - \int_\Omega (\partial_x U - U^2) (\partial_\eta \Phi)^2\ \rd(x,\eta) \\ 
& =\int_{-1}^1  (\partial_x U - U^2) \gamma^2\ \rd x + \int_\Omega U^2 (\partial_\eta \Phi)^2\ \rd(x,\eta) + 2 \int_\Omega U \omega \partial_\eta\Phi\ \rd(x,\eta) \\ 
& =\int_{-1}^1  (\partial_x U - U^2) \gamma^2\ \rd x + \int_\Omega U^2 (\partial_\eta \Phi)^2\ \rd(x,\eta) \\
& \quad + 2 \int_\Omega U \partial_\eta\Phi (\omega - \eta U \zeta)\ \rd(x,\eta) + 2 \int_\Omega U^2 \eta \partial_\eta\Phi \partial_\eta^2\Phi\ \rd(x,\eta) \ .
\end{align*}
Using once more Green's formula, we end up with
\begin{equation}
R_1 =\int_{-1}^1 \partial_x U \gamma^2\ \rd x + 2 \int_\Omega U \partial_\eta\Phi (\omega - \eta U \zeta)\ \rd(x,\eta) \ . \label{g13}
\end{equation}
Since $\alpha\in [0,1/2)$, $H^{1-\alpha}(I)$ is an algebra and it follows from the fact that $u\in\overline{S}_2^{2-\alpha}(\kappa)$ and the Lipschitz continuity of $r\mapsto (1+r)^{-1}$ in  $[\kappa-1, \infty)$ that 
$$
\|\partial_x U\|_{H^{-\alpha}(I)} \le c \| U\|_{H^{1-\alpha}(I)} \le c \|\partial_x u\|_{H^{1-\alpha}(I)} \left\| \frac{1}{1+u} \right\|_{H^{1-\alpha}(I)} \le c(\kappa) 
$$
while the continuity of pointwise multiplication (see \cite[Theorem~4.1 \& Remark~4.2(d)]{Am91})
$$
H^{1/2}(I) \cdot H^{\nu}(I) \longrightarrow H^\alpha(I)\ , \quad 0\le \alpha < \nu < \frac{1}{2}\ ,
$$
gives
$$
\left\| \gamma^2 \right\|_{H^\alpha(I)} \le c \|\gamma\|_{H^{1/2}(I)} \|\gamma\|_{H^\nu(I)}\ .
$$
Since the trace operator maps  $H^s(\Omega)$ continuously in $H^{s-1/2}(I)$ for all $s\in (1/2,1]$ by \cite[Theorem~1.5.2.1]{Gr85} and since the complex interpolation space $\left[ L_2(\Omega) , H_D^{1}(\Omega) \right]_{(2\nu+1)/2}$ coincides  up to equivalent norms with $H_D^{(2\nu+1)/2}(\Omega)$ we further obtain 
\begin{align*}
\left\| \gamma^2 \right\|_{H^\alpha(I)} & \le c \|\partial_\eta \Phi\|_{H^1(\Omega)} \|\partial_\eta \Phi\|_{H^{(2\nu+1)/2}(\Omega)} \\
& \le c \|\partial_\eta \Phi\|_{H^1(\Omega)}^{(3+2\nu)/2} \|\partial_\eta \Phi\|_{L_2(\Omega)}^{(1-2\nu)/2}\ . 
\end{align*}
We now combine the above estimates, \eqref{g13}, Young's inequality, the continuous embedding of $H^{2-\alpha}(I)$ in $W_\infty^1(I)$, and \eqref{g6} to obtain, for $\delta \in (0,1)$, 
\begin{align*}
|R_1| & \le \|\partial_x U\|_{H^{-\alpha}(I)} \left\| \gamma^2 \right\|_{H^\alpha(I)} + \frac{1}{2} \| \omega - \eta U \zeta \|_{L_2(\Omega)}^2 + 2 \|U\|_{L_\infty(I)}^2 \left\| \partial_\eta\Phi \right\|_{L_2(\Omega)}^2 \\
& \le c(\kappa) \|\partial_\eta \Phi\|_{H^1(\Omega)}^{(3+2\nu)/2} \|\partial_\eta \Phi\|_{L_2(\Omega)}^{(1-2\nu)/2} \\
& \quad + \frac{1}{2} \| \omega - \eta U \zeta \|_{L_2(\Omega)}^2 + c \|\partial_x u\|_{L_\infty(I)}^2 \left\| \frac{1}{1+u} \right\|_{L_\infty(I)}^2 \|\partial_\eta\Phi\|_{L_2(\Omega)}^2 \\
& \le \delta \|\partial_\eta \Phi\|_{H^1(\Omega)}^2 + c(\kappa,\delta) \|\partial_\eta \Phi\|_{L_2(\Omega)}^2 + \frac{1}{2} \| \omega - \eta U \zeta \|_{L_2(\Omega)}^2 \ .
\end{align*}
Since 
\begin{align}
\|\partial_\eta \Phi\|_{H^1(\Omega)} & \le \|\partial_\eta \Phi\|_{L_2(\Omega)} + \|\omega\|_{L_2(\Omega)} + \|\zeta\|_{L_2(\Omega)} \nonumber\\
& \le \|\partial_\eta \Phi\|_{L_2(\Omega)} + \|\omega-\eta U \zeta\|_{L_2(\Omega)} + \|\partial_x u\|_{L_\infty(I)} \left\| \frac{\zeta}{1+u} \right\|_{L_2(\Omega)}\nonumber \\
&\quad+ \|1+u\|_{L_\infty(I)} \left\| \frac{\zeta}{1+u} \right\|_{L_2(\Omega)} \nonumber \\
& \le \|\partial_\eta \Phi\|_{L_2(\Omega)} + \|\omega-\eta U \zeta\|_{L_2(\Omega)} + c(\kappa) \left\| \frac{\zeta}{1+u} \right\|_{L_2(\Omega)} \label{g13a}
\end{align}
and
$$
\|\partial_\eta \Phi\|_{L_2(\Omega)} \le c_3(\kappa) \|h\|_{H_{D}^{-1}(\Omega)}
$$
by \eqref{g6}, we further obtain
$$
|R_1| \le \left( 2\delta + \frac{1}{2} \right) \| \omega - \eta U \zeta \|_{L_2(\Omega)}^2 + c(\kappa) \delta \left\| \frac{\zeta}{1+u} \right\|_{L_2(\Omega)}^2 + c(\kappa,\delta) \|h\|_{H_{D}^{-1}(\Omega)}^2 \ .
$$
Choosing $\delta\in (0,1/4)$ such that $c(\kappa) \delta <1/(2\varepsilon^2)$, we conclude that
\begin{equation}
|R_1| \le \| \omega - \eta U \zeta \|_{L_2(\Omega)}^2 + \frac{1}{2\varepsilon^2} \left\| \frac{\zeta}{1+u} \right\|_{L_2(\Omega)}^2 + c(\kappa) \|h\|_{H_{D}^{-1}(\Omega)}^2\ . \label{g14}
\end{equation}
Next, by Cauchy-Schwarz' and Young's inequalities,
\begin{equation}
|R_2| \le \|1+u\|_{L_\infty(I)} \|h\|_{L_2(\Omega)} \left\| \frac{\zeta}{1+u} \right\|_{L_2(\Omega)} \le \frac{1}{4} \left\| \frac{\zeta}{1+u} \right\|_{L_2(\Omega)}^2 + c(\kappa) \|h\|_{L_2(\Omega)}^2\ . \label{g15}
\end{equation}
We then infer from \eqref{g10}, \eqref{g14}, and \eqref{g15} that
$$
\varepsilon^2 \|\omega - \eta U \zeta \|_{L_2(\Omega)}^2 + \left\| \frac{\zeta}{1+u} \right\|_{L_2(\Omega)}^2 \le c(\kappa) \left( \|h\|_{H_{D}^{-1}(\Omega)}^2 + \|h\|_{L_2(\Omega)}^2 \right) \le c(\kappa) \|h\|_{L_2(\Omega)}^2\ . 
$$
Using once more that $u\in\overline{S}_2^{2-\alpha}(\kappa)$ together with \eqref{g6} and the definition of $\omega$ and $\zeta$, we finally obtain 
\begin{equation}
\Phi\in X(\Omega)\quad\text{ with }\quad
\|\partial_\eta\Phi\|_{H^1(\Omega)} \le c(\kappa) \|h\|_{L_2(\Omega)}\ . \label{g16}
\end{equation}
Therefore, recalling the definition \eqref{g8a}, the regularity of $u$ and $\Phi$ and \eqref{g5} allow us to write
\begin{equation}
\varepsilon^2 q = \varepsilon^2 \partial_x^2 \Phi - \varepsilon^2 \eta \partial_x U \partial_\eta\Phi = h + 2 \varepsilon^2 \eta U \omega - \left[ \frac{1}{(1+u)^2} + \varepsilon^2 \eta^2 U^2 \right] \zeta - \varepsilon^2 \eta U^2\partial_\eta \Phi\ , \label{g18}
\end{equation}
and it follows from \eqref{g16} and the continuous embedding of $H^{2-\alpha}(I)$ in $W_\infty^1(I)$ that the right-hand side of the above identity belongs to $L_2(\Omega)$ with
\begin{equation}
\left\| q \right\|_{L_2(\Omega)} \le c(\kappa) \|h\|_{L_2(\Omega)}\ . \label{g17}
\end{equation}
Since
$$
\eta \partial_x U \partial_\eta\Phi  = \partial_x \left( \eta U \partial_\eta\Phi \right) - \eta U \omega
$$
and pointwise multiplication
$$
H_D^{1-\alpha}(\Omega) \cdot H_D^1(\Omega) \longrightarrow H_D^{1-\nu}(\Omega)
$$
is continuous \cite{Am91}, we deduce from \eqref{g16} and the continuous embedding of $H^{2-\alpha}(I)$ in $W_\infty^1(I)$ that 
$$
[(x,\eta)\mapsto \eta \partial_x U \partial_\eta\Phi] \in H^{-\nu}(\Omega) \;\;\text{ with }\;\; \left\| \eta \partial_x U \partial_\eta\Phi \right\|_{H^{-\nu}(\Omega)} \le c(\kappa) \|h\|_{L_2(\Omega)}\ .
$$
This last property together with \eqref{g16}, \eqref{g18}, and \eqref{g17} entails that $\Phi\in H^{2-\nu}(\Omega)$ with 
\begin{equation*}
\|\Phi\|_{H^{2-\nu}(\Omega)} \le c(\kappa) \|h\|_{L_2(\Omega)}\ .
\end{equation*}
We have thus shown that Lemma~\ref{gle3a} holds true for $u\in \overline{S}_2^{2-\alpha}(\kappa)\cap W_p^2(I)$ with $p>2$. 

\medskip

\noindent\textbf{Step~2:} Let now $u\in \overline{S}_2^{2-\alpha}(\kappa)$. Classical density arguments ensure that there is a sequence $(u_n)_{n\ge 1}$ such that $u_n\in W_3^2(I)$ for each $n\ge 1$ and
\begin{equation}
\lim_{n\to\infty} \|u_n - u \|_{H^{2-\alpha}(I)} = 0\ . \label{g19}
\end{equation}
Furthermore, owing to the continuous embedding $H^{2-\alpha}(I)$ in $W_\infty^1(I)$ and the convergence \eqref{g19}, we may assume that $u_n\in \overline{S}_2^{2-\alpha}((1+\kappa)/2)$ for each $n\ge 1$. Denoting the solution to \eqref{g5} with $u_n$ instead of $u$ by $\Phi_n$, it follows from the analysis performed in \textbf{Step~1} that $\Phi_n\in X(\Omega)\cap H^{2-\nu}(\Omega)$ satisfies  
\begin{equation}
\|\Phi_n\|_{X(\Omega)} + \|\Phi_n\|_{H^{2-\nu}(\Omega)} \le c(\kappa) \|h\|_{L_2(\Omega)}\ . \label{g20}
\end{equation}
Owing to the compactness of the embeddings of $H^{2-\nu}(\Omega)$ in $H^1(\Omega)$, Lemma~\ref{gle2c} together with \eqref{g19} and \eqref{g20} imply that
$$
\Phi_n \longrightarrow \Phi \;\text{ in }\; H^1(\Omega) \quad\text{ and }\quad \Phi_n \rightharpoonup \Phi \;\text{ in }\; X(\Omega) \cap H^{2-\nu}(\Omega)\ , 
$$
where $\Phi\in H^1_D(\Omega)$ is the weak solution to \eqref{g5} which also belongs to $X(\Omega) \cap H^{2-\nu}(\Omega)$ and satisfies \eqref{g20}.
\end{proof}

We next consider the case where the right-hand side $h$ of \eqref{g5} is less regular but is a derivative with respect to $x$.

\begin{lemma}\label{gle3b}
Let $\alpha\in [0,1/2)$, $\alpha_1\in [0,1/2)$, $\nu\in (\alpha,1/2)\cap [\alpha_1,1/2)$, $\kappa\in (0,1)$. Let $u\in \overline{S}_2^{2-\alpha}(\kappa)$ and suppose that $h\in H^{-1}(\Omega)$ is of the form 
\begin{equation}
h(x,\eta) = \partial_x h_1(x) h_2(\eta)\ , \quad (x,\eta)\in\Omega\ ,  \;\text{ with } h_1 \in H^{1-\alpha_1}(I) \;\text{ and }\; h_2 \in H^1(0,1)\ . \label{g21} 
\end{equation} 
Then the unique solution $\Phi$ to \eqref{g5}, given by Lemma~\ref{gle2}~(i), belongs to $X(\Omega)\cap H^{2-\nu}(\Omega)$ and there is $c_6(\kappa)$ depending only on $\varepsilon$, $\alpha$, $\alpha_1$, $\nu$, and $\kappa$ such that 
\begin{equation}
\|\Phi\|_{X(\Omega)} + \|\Phi\|_{H^{2-\nu}(\Omega)} \le c_6(\kappa)  \|h_1\|_{H^{1-\alpha_1}(I)}\ \|h_2\|_{H^1(0,1)}\ . \label{g23}
\end{equation}
Moreover, the distribution $q$ defined in \eqref{g8a} belongs to $H^{-\alpha_1}(\Omega)$ and there is $c_7(\kappa)$ depending only on $\varepsilon$, $\alpha$, $\alpha_1$, and $\kappa$ such that 
\begin{equation}
\| q\|_{H^{-\alpha_1}(I)} \le c_7(\kappa) \|h_1\|_{H^{1-\alpha_1}(I)}\ \|h_2\|_{H^1(0,1)}\ . \label{g23b}
\end{equation}
\end{lemma}

\begin{proof} The proof of Lemma~\ref{gle3b} follows closely that of Lemma~\ref{gle3a}, the main difference being the analysis of the terms involving $h$.

\noindent\textbf{Step~1:} We additionally assume that $u\in W_p^2(I)$ for some $p>2$ and that $h_1\in H^1(I)$. In that case the solution $\Phi$ to \eqref{g5} belongs to $H^2(\Omega)$ according to Lemma~\ref{gle2}~(ii). We then proceed as in the proof of Lemma~\ref{gle3a} and observe that \eqref{g10} as well as the estimate \eqref{g14} on $R_1$, defined in \eqref{g11}, are still valid. To estimate $R_2$, defined in \eqref{g12}, we argue differently. We use twice Green's formula to  get
\begin{align*}
R_2 & = \int_\Omega h_2 \partial_x h_1 \partial_\eta^2 \Phi\ \rd(x,\eta) \\
& =\int_{-1}^1 h_2(1) \partial_x h_1(x) \partial_\eta \Phi(x,1)\ \rd x -\int_{-1}^1 h_2(0) \partial_x h_1(x) \partial_\eta \Phi(x,0)\ \rd x - \int_\Omega \partial_x h_1 \partial_\eta h_2 \partial_\eta \Phi\ \rd(x,\eta) \\
& = h_2(1)\int_{-1}^1 \partial_x h_1(x) \partial_\eta \Phi(x,1)\ \rd x - h_2(0)\int_{-1}^1 \partial_x h_1(x) \partial_\eta \Phi(x,0)\ \rd x + \int_\Omega h_1 \partial_\eta h_2 \partial_x \partial_\eta \Phi\ \rd(x,\eta) \\
& \quad - \int_0^1 h_1(1) \partial_\eta h_2(\eta) \partial_\eta\Phi(1,\eta)\ \rd\eta + \int_0^1 h_1(-1) \partial_\eta h_2(\eta) \partial_\eta\Phi(-1,\eta)\ \rd\eta\ . 
\end{align*}
Recalling that $\Phi(1,\eta)=\Phi(-1,\eta)=0$ for $\eta\in (0,1)$ due to \eqref{g5}, we realize that the last two terms on the right-hand side of the above identity vanish and thus
$$
R_2 = h_2(1)\int_{-1}^1 \partial_x h_1(x) \partial_\eta \Phi(x,1)\ \rd x - h_2(0)\int_{-1}^1 \partial_x h_1(x) \partial_\eta \Phi(x,0)\ \rd x + \int_\Omega h_1 \partial_\eta h_2 \partial_x \partial_\eta \Phi\ \rd(x,\eta)\ .
$$
Using again the notation $U=\partial_x \ln{(1+u)}$, $\omega=\partial_x\partial_\eta \Phi$, and $\zeta = \partial_\eta^2 \Phi$, we deduce from the continuity of the trace operator from $H^1(\Omega)$ to $H^{\alpha_1}(I)$ and the continuous embedding of $H^{1-\alpha_1}(I)$ in $L_\infty(I)$ that
\begin{align*}
|R_2| & \le |h_2(1)| \|\partial_x h_1\|_{H^{-\alpha_1}(I)} \|\partial_\eta \Phi(.,1)\|_{H^{\alpha_1}(I)} + |h_2(0)| \|\partial_x h_1\|_{H^{-\alpha_1}(I)} \|\partial_\eta \Phi(.,0)\|_{H^{\alpha_1}(I)} \\
& \quad + \|h_1\|_{L_\infty(I)} \|\partial_\eta h_2\|_{L_2(0,1)} \|\omega\|_{L_2(\Omega)} \\
& \le c \|h_2\|_{H^1(0,1)} \ \|h_1\|_{H^{1-\alpha_1}(I)}\ \left( \|\partial_\eta \Phi\|_{H^1(\Omega)} + \|\omega - \eta U \zeta \|_{L_2(\Omega)} + \|\partial_x u\|_{L_\infty(I)} \left\| \frac{\zeta}{1+u} \right\|_{L_2(\Omega)} \right)\ .
\end{align*}
Since 
$$
\|\partial_\eta \Phi\|_{L_2(\Omega)} \le c(\kappa) \|h\|_{H_{D}^{-1}(\Omega)} \le c(\kappa) \|h_1 h_2\|_{L_2(\Omega)}
$$
by Lemma~\ref{gle2}~(i), we deduce from \eqref{g13a} that
$$
|R_2| \le c(\kappa) \|h_1\|_{H^{1-\alpha_1}(I)} \ \|h_2\|_{H^1(0,1)} \left( \| h_1 h_2\|_{L_2(\Omega)} + \|\omega-\eta U \zeta\|_{L_2(\Omega)} + \left\| \frac{\zeta}{1+u} \right\|_{L_2(\Omega)} \right)\ .
$$
Young's inequality finally gives
\begin{equation}
|R_2| \le \delta  \|\omega-\eta U \zeta\|_{L_2(\Omega)}^2 + \delta \left\| \frac{\zeta}{1+u} \right\|_{L_2(\Omega)}^2 + c(\kappa,\delta) \|h_1\|_{H^{1-\alpha_1}(I)}^2 \ \|h_2\|_{H^1(0,1)}^2  \label{g24}
\end{equation}
for $\delta\in (0,1)$. Choosing $\delta$ appropriately small in \eqref{g24}, we derive from \eqref{g10}, \eqref{g14}, and \eqref{g24} that 
\begin{equation*}
\begin{split}
\varepsilon^2 \|\omega - \eta U \zeta \|_{L_2(\Omega)}^2 + \left\| \frac{\zeta}{1+u} \right\|_{L_2(\Omega)}^2 &\le c(\kappa) \left( \|h\|_{H_{D}^{-1}(\Omega)}^2 + \|h_1\|_{H^{1-\alpha_1}(I)}^2 \|h_2\|_{H^1(0,1)}^2 \right)\\
& \le c(\kappa) \|h_1\|_{H^{1-\alpha_1}(I)}^2 \|h_2\|_{H^1(0,1)}^2\ .
\end{split} 
\end{equation*}
Therefore, since $u\in \overline{S}_2^{2-\alpha}(\kappa)$, we conclude as in the proof of Lemma~\ref{gle3a} that $\Phi$ belongs to $X(\Omega)$ with 
\begin{equation}
\|\Phi\|_{X(\Omega)} \le c(\kappa) \|h_1\|_{H^{1-\alpha_1}(I)} \|h_2\|_{H^1(0,1)}\ . \label{g25}
\end{equation}
Recalling the definition \eqref{g8a} and arguing as in the proof of \eqref{g17}, we infer from \eqref{g5} and \eqref{g25} that 
\begin{equation}
\|\varepsilon^2 q - h \|_{L_2(\Omega)} \le c(\kappa) \|h_1\|_{H^{1-\alpha_1}(I)} \|h_2\|_{H^1(I)}\ . \label{g26}
\end{equation}
On the one hand, the regularity \eqref{g21} of $h$ ensures that $h\in H^{-\alpha_1}(\Omega)$ and we deduce from \eqref{g26} that
\begin{equation}
\|q\|_{H^{-\alpha_1}(\Omega)} \le c(\kappa) \|h_1\|_{H^{1-\alpha_1}(I)} \|h_2\|_{H^1(0,1)}\ . \label{g27}
\end{equation}
On the other hand, arguing as in the proof of Lemma~\ref{gle3a}, we obtain from \eqref{g24} that 
$$
[(x,\eta)\mapsto \eta \partial_x U \partial_\eta\Phi] \in H^{-\nu}(\Omega) \;\;\text{ with }\;\; \left\| \eta \partial_x U \partial_\eta\Phi \right\|_{H^{-\nu}(\Omega)} \le c(\kappa) \|h_1\|_{H^{1-\alpha_1}(I)} \|h_2\|_{H^1(0,1)}\ ,
$$
while the regularity \eqref{g21} of $h$ and the choice of $\nu\ge \alpha_1$ entail that $h\in H^{-\nu}(\Omega)$. We combine these facts with \eqref{g25} and \eqref{g27} to conclude that $\Phi\in H^{2-\nu}(\Omega)$ satisfies
$$
\|\Phi\|_{H^{2-\nu}(\Omega)} \le c(\kappa) \|h_1\|_{H^{1-\alpha_1}(\Omega)} \|h_2\|_{H^1(0,1)}\ .
$$
We have thereby established Lemma~\ref{gle3b} for all functions $u\in\overline{S}_2^{2-\alpha}(\kappa)$ and $h\in H^{-1}(\Omega)$ satisfying \eqref{g21} under the additional assumption that $u\in W_p^2(I)$ and $h_1\in H^1(I)$.  

\noindent\textbf{Step~2:} We now consider $u\in\overline{S}_2^{2-\alpha}(\kappa)$ and $h\in H^{-\nu}(\Omega)$ satisfying \eqref{g21}. Classical approximation arguments guarantee that there are sequences $(u_n)_{n\ge 1}$ in $ W_3^2(I)$ and $(h_{1,n})_{n\ge 1}$ in $H^1(I)$ such that
$$
\lim_{n\to \infty} \|u_n - u \|_{H^{2-\alpha}(I)} = \lim_{n\to \infty} \|h_{1,n} - h_1 \|_{H^{1-\alpha_1}(I)} = 0\ .
$$
We then proceed as in the second step of the proof of Lemma~\ref{gle3a} to complete the proof of Lemma~\ref{gle3b}. 
\end{proof}

\subsection{Proof of Proposition~\ref{gth1} and Corollary~\ref{gth0}}\label{subsec:42}

We are now in a position to complete the proof of Proposition~\ref{gth1} by considering the particular right-hand side $f_u$ of \eqref{g2} given in \eqref{ff}. For the remainder of this subsection, we set 
$$
U(x) := \frac{\partial_x u(x)}{1+u(x)}\ ,\quad x\in I\ ,
$$
so that
$$
f_u(x,\eta) = \varepsilon^2 \eta \left[ \partial_x U(x) - U(x)^2 \right]\ , \quad (x,\eta)\in \Omega\ .
$$

\begin{proof}[Proof of Proposition~\ref{gth1}]
Let $u\in \overline{S}_2^{2-\alpha}(\kappa)$.
We handle the cases $\alpha=0$ and $\alpha\in (0,1/2)$ separately.

\smallskip

\noindent\textbf{Case~1: $\alpha=0$.} In that case, $u\in H^2(I)$ from which we readily infer that
\begin{equation*}
f_u \in L_2(\Omega) \quad \text{ and }\quad \|f_u\|_{L_2}\le c(\kappa)\ . 
\end{equation*} 
Fix $\nu\in (0,1/2)$. It follows from Lemma~\ref{gle2} and Lemma~\ref{gle3a} with $h=f_u$ that \eqref{g2} has a unique solution $\Phi_u\in X(\Omega)\cap H^{2-\nu}(\Omega)$ which satisfies \eqref{g3}. Moreover, the distribution $q_u$ defined by \eqref{g1a} belongs to $L_2(\Omega)$ according to Lemma~\ref{gle3a}, and \eqref{g3a} follows from \eqref{g8b}. 

Now, if $(u_n)_{n\ge 1}$ is a sequence in $\overline{S}_2^{2}(\kappa)$ converging weakly in $H^{2}(I)$ toward $u\in \overline{S}_2^{2}(\kappa)$, the compactness of the embedding of $H^2(I)$ in $W_\infty^1(I)$ entails that $(f_{u_n})_{n\ge 1}$ converges weakly toward $f_u$ in $L_2(\Omega)$. Hence, due to Lemma~\ref{gle2c}, $(\Phi_{u_n})_{n\ge 1}$ converges weakly toward $\Phi_u$ in $H^1(\Omega)$. Since $(\Phi_{u_n})_{n\ge 1}$ is actually bounded in $X(\Omega)\cap H^{2-\nu}(\Omega)$ by \eqref{g3}, the above convergence can readily be improved to \eqref{g4}. The compactness of the embedding of $H^{2-\nu}(\Omega)$ in $H^1(\Omega)$ finally guarantees the strong convergence of $(\Phi_{u_n})_{n\ge 1}$ toward $\Phi_u$ in $H^1(\Omega)$.

\smallskip

\noindent\textbf{Case~2: $\alpha\in (0,1/2)$.}
In that case the space $H^{1-\alpha}(I)$ is an algebra so that both $U$ and $U^2$ belong to $H^{1-\alpha}(I)$. Introducing
$$
f_1(x) := \varepsilon^2 \left[ U(x) - \int_0^x U^2(x')\ \rd x' \right]\ , \quad x\in I\ ,
$$
we realize that 
\begin{equation*}
f_u(x,\eta) = \eta \partial_x f_1(x) \quad \text{ with }\quad \|f_1\|_{H^{1-\alpha}}\le c(\kappa) 
\end{equation*} 
for some positive constant $c(\kappa)$ depending only on $\varepsilon$, $\alpha$, and $\kappa$. Fix $\nu\in (\alpha,1/2)$. We infer from Lemma~\ref{gle2} and Lemma~\ref{gle3b} with $h=f_u$ and $\alpha_1=\alpha$ that \eqref{g2} has a unique solution $\Phi_u\in X(\Omega)\cap H^{2-\nu}(\Omega)$ which satisfies \eqref{g3}. Also the distribution $q_u$ defined in \eqref{g1a} belongs to $H^{-\alpha}(\Omega)$ by Lemma~\ref{gle3b}, and \eqref{g3a} follows from \eqref{g23b}. Finally, the proof of the continuity property stated in Proposition~\ref{gth1} is the same as in the previous case $\alpha=0$.
\end{proof}

Finally, we may apply the information gathered on the equation \eqref{g2} for $\Phi_u$ to the problem \eqref{psieq}-\eqref{psibc} for $\psi_u$ and prove  Corollary~\ref{gth0}.

\begin{proof}[Proof of Corollary~\ref{gth0}]
Let $\alpha\in [0,1/2)$ and $u\in S^{2-\alpha}$. Since $H^{2-\alpha}(I)$ embeds continuously in $C([-1,1])$ there clearly is some $\kappa\in (0,1)$ such that $u\in S_2^{2-\alpha}(\kappa)$. Let $\Phi_u$ and $\psi_u$ be the unique solution to \eqref{g2} and respectively \eqref{psieq}-\eqref{psibc} and recall that
$$
\psi_u(x,z) = \Phi_u(x,\eta)+\eta
$$
for $(x,z)\in\Omega(u)$ and $(x,\eta)\in\Omega$ with $(x,z)=(x,-1+(1+u(x))\eta)$. Straightforward computations then give
\begin{eqnarray*}
\partial_x^2 \psi_u(x,z) & = & \partial_x^2 \Phi_u(x,\eta) - \eta \partial_x U(x) \partial_\eta \Phi_u(x,\eta) - 2 \eta U(x) \partial_x \partial_\eta \Phi_u(x,\eta) \\
& & \  + \eta^2 U(x)^2 \partial_\eta^2 \Phi_u(x,\eta) + \eta U(x)^2 \partial_\eta \Phi_u(x,\eta) + \eta \left[ U^2 - \partial_x U \right] (x)\ , \\
\partial_x \partial_z \psi_u(x,z) & = & \frac{1}{1+u(x)} \partial_x \partial_\eta \Phi_u(x,\eta) - \eta \frac{U(x)}{1+u(x)} \partial_\eta^2 \Phi_u(x,\eta) - \frac{U(x)}{1+u(x)} \left[ 1 + \partial_\eta \Phi_u(x,\eta) \right]\ , \\
\partial_z^2 \psi_u(x,z) & = & \frac{1}{(1+u(x))^2} \partial_\eta^2 \Phi_u(x,\eta)\ ,
\end{eqnarray*}
where $U:= \partial_x\ln{(1+u)}$. It readily follows from the regularity of $u$ and Proposition~\ref{gth1} that $\partial_x \partial_z \psi_u$ and $\partial_z^2 \psi_u$ both belong to $L_2(\Omega(u))$. As for $\partial_x^2 \psi_u$, it also reads
$$
\partial_x^2 \psi_u= q_u + r_u + s_u 
$$
with
\begin{align*}
r_u(x,\eta) & :=- 2 \eta U(x) \partial_x \partial_\eta \Phi_u(x,\eta)  + \eta^2 U(x)^2 \partial_\eta^2 \Phi_u(x,\eta) + \eta U(x)^2 \partial_\eta \Phi_u(x,\eta) \ , \\
s_u (x,\eta) & := \eta \left[ U(x)^2 - \partial_x U(x) \right]\ ,
\end{align*} 
for $(x,\eta)\in \Omega$, the distribution $q_u$ being defined in \eqref{g1a}. The regularity of $u$ and Proposition~\ref{gth1} imply $r_u\in L_2(\Omega)$ while the distributions $q_u$ and $s_u$ both belong to $H^{-\alpha}(\Omega)$. Consequently, $\psi_u\in  H^{2-\alpha}(\Omega(u))$.

As for the continuity of $g$ recall that $g(u)$ may be written alternatively as
\begin{equation*}
g(u)(x)=\frac{1+\varepsilon^2\vert\partial_x u(x)\vert^2}{(1+u(x))^2}\vert\partial_\eta\Phi_u(x,1)\vert^2\ ,\quad x\in I\  .
\end{equation*}
Let $(u_n)_{n\ge 1}$ be any sequence in $S_2^{2-\alpha}(\kappa)$ with $u_n\rightarrow u$ in $H^{2-\alpha}(I)$. Then, for each $s\in (0,1/2)$, the convergence \eqref{g4} and the compactness of the embedding of $H^1(\Omega)$ in $H^{1-s}(\Omega)$ imply that
$$
\partial_\eta\Phi_{u_n}\rightarrow \partial_\eta \Phi_u\ \text{ in } \ H^{1-s}(\Omega)
$$
and thus, according to \cite[Theorem 1.5.1.2]{Gr85},
$$
\partial_\eta\Phi_{u_n}(\cdot,1)\rightarrow \partial_\eta \Phi_u(\cdot,1)\ \text{ in } \ H^{1/2-s}(I)\ .
$$
Since pointwise multiplication
$$
H^{1-\alpha}(I)\cdot H^{1-\alpha}(I)\cdot H^{1/2-s}(I)\cdot H^{1/2-s}(I) \hookrightarrow H^{\sigma}(I)
$$
is continuous for each $\sigma\in [0,1/2-2s)$ according to  \cite[Theorem~4.1 \& Remark~4.2(d)]{Am91}, we conclude that $g(u_n)\rightarrow g(u)$ in $H^\sigma(I)$ and thus the continuity of $g : S_2^{2-\alpha}(\kappa)\rightarrow H^\sigma(I)$ for all $\sigma\in [0,1/2)$ as $s\in (0,1/2)$ is arbitrary.
\end{proof}




\begin{thebibliography}{99.}%
%
\bibitem{Am91}
H.~Amann. Multiplication in Sobolev and Besov spaces. In {\em Nonlinear analysis}, Scuola Norm. Sup. di Pisa Quaderni, 27--50, (Scuola Norm. Sup., Pisa, 1991)
%
\bibitem{Bo05}
T.~Boggio. Sulle funzioni di Green d'ordine $m$, Rend. Circ. Mat. Palermo \textbf{20}, 97--135 (1905)
%
\bibitem{ELW13}
J.~Escher, Ph.~Lauren\c{c}ot, Ch.~Walker. Finite time singularity in a free boundary problem modeling MEMS, C. R. Acad. Sci. Paris S\'er. I Math. \textbf{351}, 807--812 (2013)
%
\bibitem{ELW14}
J.~Escher, Ph.~Lauren\c{c}ot, Ch.~Walker. A parabolic free boundary problem modeling electrostatic MEMS, Arch. Ration. Mech. Anal. \textbf{211}, 389--417 (2014)
%
\bibitem{ELWxx}
J.~Escher, Ph.~Lauren\c{c}ot, Ch.~Walker. Dynamics of a free boundary problem with curvature modeling electrostatic MEMS. Trans. Amer. Math. Soc. (to appear)
%

\bibitem{EGG10}
P.~Esposito, N.~Ghoussoub, Y.~Guo.
\newblock {\em Mathematical Analysis of Partial Differential Equations Modeling Electrostatic MEMS}, 
\newblock Courant Lecture Notes in Mathematics {\bf 20}, (Courant Institute of Mathematical Sciences, New York, 2010)
%
\bibitem{GT01}
D.~Gilbarg, N.S.~Trudinger. \textit{Elliptic Partial Differential Equations of Second Order}, Reprint of 1998 edition, Classics in Mathematics (Springer-Verlag, Berlin, 2001)
%
\bibitem{Gr85}
P.~Grisvard. \textit{Elliptic Problems in Nonsmooth Domains}, Monographs and Studies in Mathematics \textbf{24} (Pitman, Boston, 1985)
%
\bibitem{Gr02}
H.-C.~Grunau. Positivity, change of sign and buckling eigenvalues in a one-dimensional fourth order model problem, Adv. Differential Equations \textbf{7}, 177--196 (2002)
%
\bibitem{HP05}
A.~Henrot, M.~Pierre. \textit{Variation et Optimisation de Formes}, Math\'ematiques $\&$ Applications (Berlin) \textbf{48} (Springer, Berlin, 2005)
%
\bibitem{LW13}
Ph.~Lauren\c{c}ot, Ch.~Walker. A stationary free boundary problem modeling electrostatic MEMS, Arch. Ration. Mech. Anal.  \textbf{207}, 139--158 (2013) 
%
\bibitem{LWxx}
Ph.~Lauren\c{c}ot, Ch.~Walker. A free boundary problem modeling electrostatic MEMS: I. Linear bending effects, Math. Ann. (to appear)
%
\bibitem{LWxy}
Ph.~Lauren\c{c}ot, Ch.~Walker. Sign-preserving property for some fourth-order elliptic operators in one dimension or in radial symmetry, J. Anal. Math. (to appear)
%
\bibitem{LWxxx}
Ph.~Lauren\c{c}ot, Ch.~Walker. A free boundary problem modeling electrostatic MEMS: II. Nonlinear bending effects, Math. Mod. Meth. Appl. Sci. (to appear)
%
\bibitem{LWxxz}
Ph.~Lauren\c{c}ot, Ch.~Walker. A fourth-order model for MEMS with clamped boundary conditions, Proc. London Math. Soc. (to appear)
%
\bibitem{Ow97}
M.~P. Owen. Asymptotic first eigenvalue estimates for the biharmonic operator on a rectangle, J. Differential Equations \textbf{136}, 166--190 (1997)
%
\bibitem{PB03}
J.A.~Pelesko, D.H.~Bernstein. \textit{Modeling MEMS and NEMS} (Chapman \& Hall/CRC, Boca Raton, 2003)
%
\bibitem{Sv93}
V.~{\v S}ver\'ak. On optimal shape design, J. Math. Pures Appl. \textbf{72}, 537--551 (1993)
%
\bibitem{Ze86}
E.~Zeidler. \textit{Nonlinear Functional Analysis and its Applications: I: Fixed-Point Theorems} (Springer, 1986)
%
\bibitem{Ze95}
E.~Zeidler. \textit{Applied Functional Analysis. Main Principles and Their Applications} (Springer, New York 1995)



\end{thebibliography}
\end{document}